\documentclass[a4paper,11pt]{article}
\usepackage[english]{babel}
\usepackage[colorlinks,citecolor=green,linkcolor=red]{hyperref}
\usepackage{amsthm}
\usepackage{color}
\usepackage{mathrsfs}
\usepackage{amsmath}
\usepackage{mathtools}
\usepackage{amsfonts}
\usepackage{amssymb}
\usepackage{bm}
\usepackage{physics}
\usepackage{enumitem}
\usepackage{tikz-cd}
\usepackage{a4wide}
\usepackage{cleveref}
\usepackage{esint}
\usepackage{nicefrac}
\usepackage{yfonts}
\usepackage[toc]{multitoc}

\let\oldbibliography\thebibliography
\renewcommand{\thebibliography}[1]{
\oldbibliography{#1}
\setlength{\itemsep}{0pt}}

\numberwithin{equation}{section}

\newtheorem{thm}{Theorem}[section]
\newtheorem*{thm*}{Theorem}
\newtheorem{lem}[thm]{Lemma}
\newtheorem{prop}[thm]{Proposition}

\theoremstyle{definition}
\newtheorem{defn}[thm]{Definition}

\theoremstyle{remark}
\newtheorem{rem}[thm]{Remark}

\newcommand{\lip}{{\mathrm {lip}}}

\newcommand{\DIFF}{{\mathrm{D}}}

\newcommand{\heat}{{\mathsf {h}}}

\newcommand{\capa}{{\mathrm {Cap}}}

\newcommand{\mres}{\mathbin{\vrule height 1.6ex depth 0pt width 0.13ex\vrule height 0.13ex depth 0pt width 1.3ex}}

\newcommand{\dist}{{\mathsf{d}}}

\newcommand{\mass}{{\mathsf{m}}}

\newcommand{\XX}{{\mathsf{X}}}
\newcommand{\YY}{{\mathsf{Y}}}
\newcommand{\ZZ}{{\mathsf{Z}}}
\newcommand{\FF}{{\mathcal{F}}}

\newcommand{\defeq}{\mathrel{\mathop:}=}

\newcommand{\HH}{\mathcal{H}}
\newcommand{\LL}{\mathcal{L}}

\newcommand{\RR}{\mathbb{R}}
\newcommand{\NN}{\mathbb{N}}

\newcommand{\LIP}{{\mathrm {LIP}}}
\newcommand{\BV}{{\mathrm {BV}}}

\newcommand{\RCD}{{\mathrm {RCD}}}
\newcommand{\CD}{{\mathrm {CD}}}
\newcommand{\PI}{{\mathrm {PI}}}

\def\XXint#1#2#3{{\setbox0=\hbox{$#1{#2#3}{\int}$}
		\vcenter{\hbox{$#2#3$}}\kern-.5\wd0}}

\newcommand{\fr}{\penalty-20\null\hfill$\blacksquare$}

\newcommand{\nchi}{{\raise.3ex\hbox{\(\chi\)}}}

\title{Sobolev and BV functions on $\mathrm{RCD}$ spaces \\ via the short-time behaviour of the heat kernel}

\author{Camillo Brena\footnote{\href{mailto:camillo.brena@sns.it}{camillo.brena@sns.it}, Scuola Normale Superiore, Piazza dei Cavalieri, 7, 56126 Pisa, Italy.}
\and Enrico Pasqualetto\footnote{\href{mailto:enrico.e.pasqualetto@jyu.fi}{enrico.e.pasqualetto@jyu.fi}, Department of Mathematics and Statistics, P.O.\ Box 35 (MaD), FI-40014 University of Jyvaskyla, Finland.}
\and Andrea Pinamonti\footnote{\href{mailto:andrea.pinamonti@unitn.it}{andrea.pinamonti@unitn.it}, Dipartimento di Matematica, Universit\`{a} degli Studi di Trento, Via Sommarive 14, 38123 Povo (Trento), Italy.}}

\begin{document}

\maketitle

\begin{abstract}
In the setting of finite-dimensional $\mathrm{RCD}(K,N)$ spaces, we characterize the $p$-Sobolev spaces for $p\in(1,\infty)$ and the space of functions of bounded variation in terms of the short-time behaviour of the heat flow.
Moreover, we prove that Cheeger $p$-energies and total variations can be computed as limits of nonlocal functionals involving the heat kernel.
\end{abstract}


\tableofcontents

\section{Introduction}
\subsection*{General overview}
The theory of metric measure spaces verifying synthetic lower Ricci curvature bounds has been the object of impressive developments
in the last few years. After the Curvature-Dimension condition (\(\CD\) for short) was introduced by Lott--Villani \cite{Lott-Villani09} and Sturm \cite{Sturm06I,Sturm06II}
independently, the subclass of the so-called \(\RCD\) spaces -- which select the Riemannian-type structures among the Finslerian ones --
was extensively studied starting from \cite{Ambrosio_2014,Gigli12,AmbrosioGigliMondinoRajala12}. Besides smooth Riemannian manifolds (possibly weighted and/or with convex boundary)
whose generalized Ricci tensor is bounded from below, the class of \(\RCD(K,N)\) spaces includes finite-dimensional Alexandrov spaces with sectional curvature bounded from below
\cite{Petrunin11}, as well as Ricci limit spaces \cite{Cheeger-Colding97I,Cheeger-Colding97II,Cheeger-Colding97III}, which arise as Gromov--Hausdorff limits of
Riemannian manifolds having constant dimension and Ricci curvature uniformly bounded from below. The interested reader is referred to the survey \cite{AmbICM} for an account
of this theory and a detailed bibliography.
\medskip

In this paper, we focus on the analysis of \(p\)-Sobolev functions with \(p\in(1,\infty)\), and of functions of bounded variation (BV for short), in the framework
of finite-dimensional \(\RCD\) spaces. Even though meaningful notions of Sobolev functions \cite{Cheeger00,Shanmugalingam00,AmbrosioGigli11} and of BV functions
\cite{MIRANDA2003,ADM2014} are available on arbitrary metric measure spaces, much more refined results were obtained in the setting of PI spaces (i.e.\ doubling spaces
supporting a weak form of Poincar\'{e} inequality). However, in the smaller class of \(\RCD(K,N)\) spaces an even more refined Sobolev and  BV calculus -- closely resembling
the one in the Euclidean or Riemannian setting -- is possible, as shown e.g.\ by
\cite{Gigli14,ambrosio2018rigidity,bru2019rectifiability,BG22,ABP22}. The aim of the present paper
is to characterize Sobolev and BV functions, as well as their Sobolev and BV norms, on \(\RCD(K,N)\) spaces, by means of the short-time behaviour of a nonlocal
functional involving the heat kernel. See Theorems \ref{main1} and \ref{main2} for the precise statements.
\medskip

In their seminal work \cite{BBM}, Bourgain--Brezis--Mironescu showed that if \(\Omega\subseteq\RR^n\) is a smooth bounded domain and
\(p\in(1,\infty)\), then the \(p\)-Sobolev seminorm of a function \(f\in L^p(\Omega)\) coincides (up to
a multiplicative factor, depending only on \(p\) and \(n\)) with the limit
\[
\lim_{i\to\infty}\bigg(\int_\Omega\int_\Omega\frac{|f(x)-f(y)|^p}{|x-y|^p}\rho_i(|x-y|)\,\dd x\,\dd y\bigg)^{1/p},
\]
where \((\rho_i)_{i\in\NN}\) are suitably chosen kernels of mollification. The result was then generalized to BV functions
by D\'{a}vila \cite{Dav} and Ambrosio independently. In the BV case, it holds that if \(\Omega\subseteq\RR^n\) is a bounded domain with Lipschitz
boundary, then the total variation \(|\DIFF f|(\Omega)\) of any function \(f\in L^1(\Omega)\) coincides (up to a multiplicative
factor, depending only on \(n\)) with
\[
\lim_{i\to\infty}\int_\Omega\int_\Omega\frac{|f(x)-f(y)|}{|x-y|}\rho_i(|x-y|)\,\dd x\,\dd y.
\]
Later on, several different generalizations of these results in Euclidean spaces were considered: e.g.\ by Ponce \cite{Pon},
Leoni--Spector \cite{LS1,LS2}, Nguyen \cite{Ngu}, Brezis--Nguyen \cite{BNgu,BNgu2,BNgu3,BNgu4}, Pinamonti--Vecchi--Squassina \cite{PVS1},
Nguyen--Pinamonti--Vecchi--Squassina \cite{NPVS}, Comi-Stefani \cite{CSte}, Brezis--Van Schaftingen--Yung \cite{BVY,BVY2,BVY3},
Garofalo--Tralli \cite{GarTral1, GarTral2}, Maalaoui-Pinamonti \cite{MalPin}.
\medskip

Brezis suggested in \cite[Remark 6]{Bre2} that it might be interesting to generalize the theory to more general metric measure spaces
\((\XX,\dist,\mass)\). In this direction, the first contribution was given by Di Marino--Squassina \cite{DiSqua}, who worked in the
setting of \(p\)-PI spaces and with a specific family of mollifiers. Later on Gorny \cite{Gor} and Han--Pinamonti \cite{HanPin} proved
the equality for $1<p<\infty$ for a suitable class of metric measure spaces and mollifiers. Recently, in \cite{LaPiZh22} the authors
studied a characterization of BV in metric measure spaces supporting a doubling measure and a Poincar\'e inequality for a large class of mollifiers.
In the same paper, the authors also provided a counterexample in the case $p = 1$, demonstrating that (unlike in Euclidean spaces) in metric measure
spaces the limit of the nonlocal functionals is only comparable, but not necessarily equal, to the total variation measure of the function $f$.
\medskip

In many cases of interest, a good choice of mollifiers is given by the (suitably rescaled) heat kernels. In the case of functions of bounded variation,
this has been done by several authors: by Miranda Jr.--Pallara--Paronetto--Preunkert on Euclidean spaces \cite{MPPP}
and on Riemannian manifolds \cite{MPPPbis}, by Carbonaro--Mauceri on Riemannian manifolds \cite{carbonaro_mauceri_2007},
by Bramanti--Miranda Jr.-Pallara on Carnot groups \cite{BMP}, by Marola--Miranda--Shanmugalingam on PI spaces \cite{MMS16}.
\subsection*{Statement of results}
Let us now state the main results of this paper. In an \(\RCD(K,N)\) space \((\XX,\dist,\mass)\), we denote by \((0,+\infty)\times\XX\times\XX\ni(t,x,y)\mapsto p_t(x,y)\) the heat kernel
and by \({\rm Ch}_p(f)=\int|\dd_p f|^p\,\dd\mass\) the Cheeger \(p\)-energy of a function \(f\), which has to be understood as $+\infty$ if $f\notin W^{1,p}(\XX)$. Concerning Sobolev spaces, our main result is the following:
\begin{thm}\label{main1}
Let \((\XX,\dist,\mass)\) be an \(\RCD(K,N)\) space, for some \(K\in\RR\) and \(N\in[1,\infty)\). Let $p\in(1,\infty)$ and $f\in L^p(\mass)$ be given. 
Then, denoting by \(\Gamma\) the Euler's gamma function, it holds
\begin{equation}\label{eq:claim_main1}
\exists\lim_{t\searrow 0}\frac{1}{t^{p/2}}\int\!\!\!\int p_t(x,y)|f(x)-f(y)|^p\,\dd\mass(x)\,\dd\mass(y)=\frac{2^p}{\sqrt{\pi}}\Gamma\bigg(\frac{p+1}{2}\bigg){\rm Ch}_p(f).
\end{equation}
\end{thm}
In the BV case, our main result (calling \(|\DIFF f|\) the total variation measure, which has to be understood as identically equal to $+\infty$ if $f\notin\BV(\XX)$) is the following:
\begin{thm}\label{main2}
Let \((\XX,\dist,\mass)\) be an \(\RCD(K,N)\) space, for some \(K\in\RR\) and \(N\in[1,\infty)\). Let $f\in L^1(\mass)$ be given. 
Then it holds that
\begin{equation}\label{eq:claim_main3minus1}
\exists\lim_{t\searrow 0}\frac{1}{\sqrt t}\int\!\!\!\int p_t(x,y)|f(x)-f(y)|\,\dd\mass(x)\,\dd\mass(y)=\frac{2}{\sqrt{\pi}}|\DIFF f|(\XX).
\end{equation}

\end{thm}
We will also obtain the following result; compare with \cite[Theorem 4.3]{MPPP}.
\begin{thm}\label{thm:main3}
Let \((\XX,\dist,\mass)\) be an \(\RCD(K,N)\) space, for some \(K\in\RR\) and \(N\in[1,\infty)\).
Let $f,g\in\BV(\XX)$ be such that $g\in L^\infty(\mass)$. Then
\begin{equation}\label{eq:claim_main3}
	\lim_{t\searrow 0}\frac{1}{\sqrt t}\int (f-\heat_t f) g\,\dd\mass=\frac{1}{\sqrt\pi}\frac{\omega_{n-1}}{\omega_n}\int_{S_f\cap S_g} (f^\vee-f^\wedge)(g^\vee-g^\wedge)(\nu_f\cdot\nu_g)\,\dd\HH^h.
\end{equation}
\end{thm}
In the above formula, we denote by \(\omega_n\) the Lebesgue measure of the unit ball in \(\RR^n\),
by \(\{\heat_t\}_{t>0}\) the heat flow semigroup, by \(\nu_f\) the polar vector field of \(f\)
(see Theorem \ref{Gauss-Green}), and by \(\HH^h\) the codimension-one Hausdorff measure in \(\XX\)
(see \eqref{eq:cod_1_Hausd}). Moreover, the \emph{approximate lower limit} \(f^\wedge\) and the
\emph{approximate upper limit} \(f^\vee\) are the Borel functions on \(\XX\) given by
\[\begin{split}
f^\wedge(x)&\coloneqq\sup\bigg\{t\in[-\infty,+\infty]\;\bigg|\;
\lim_{r\searrow 0}\frac{\mass(B_r(x)\cap\{f<t\})}{\mass(B_r(x))}=0\bigg\},\\
f^\vee(x)&\coloneqq\inf\bigg\{t\in[-\infty,+\infty]\;\bigg|\;
\lim_{r\searrow 0}\frac{\mass(B_r(x)\cap\{f>t\})}{\mass(B_r(x))}=0\bigg\},
\end{split}\]
respectively, while the Borel set \(S_f\subseteq\XX\) is defined as \(S_f\coloneqq\{x\in\XX\,:\,f^\wedge(x)<f^\vee(x)\}\).
\subsection*{Strategy of the proof}
Without entering into details, we sketch the lines of reasoning that will lead to the proof of our main results Theorem \ref{main1} and Theorem \ref{main2}. The first step is to show that if the $\liminf_{t\searrow 0}$ of the quantity appearing at the left-hand side of \eqref{eq:claim_main1} (resp.\ \eqref{eq:claim_main3minus1}) is finite, then $f\in W^{1,p}(\XX)$ (resp.\ $f\in\BV(\XX)$). This is an immediate consequence of results already present in the literature (\cite{GT21,MMS16}), combined with the Gaussian lower bound for the heat kernel \eqref{eq:heat_kernel_ptwse_est}. Then the bulk of the proof is to show that if $f\in W^{1,p}(\XX)$ (resp.\ $f\in\BV(\XX)$), then the limit of the quantity appearing at the left-hand side of \eqref{eq:claim_main1} (resp.\ \eqref{eq:claim_main3minus1}) equals the right-hand side of \eqref{eq:claim_main1} (resp.\ \eqref{eq:claim_main3minus1}). We briefly discuss the two cases separately, as for technical details the proofs are very different, although both can be seen, in principle, as the reduction of the claim to a particular class of functions followed by the proof of the claim for this particular class of functions via a blow-up argument.

For what concerns Sobolev functions, we borrow the argument of \cite{Gor}. 
Namely, we first exploit the density in norm of Lipschitz functions in the Sobolev space and a continuity argument to prove that it is enough to show the claim for this kind of functions (Lemma \ref{Gorny1}). Then dominated convergence will allow us to exploit the blow-up Lemma \ref{Gorny2} and thus conclude easily. 
We also stress that the blow-up argument builds upon the theory developed in \cite{Cheeger00}, as well as the convergence of the heat flow under pmGH convergence of $\RCD$ spaces. We remark that the latter fact is linked to the Mosco convergence of the Cheeger energies and relies on the curvature bound \cite{GMS15}. Due to a different choice of the mollifiers involved with respect to \cite{Gor}, we have to use with care the Gaussian estimates for the heat kernel.
 
In the case of functions of bounded variation, we do not have density of Lipschitz functions with respect to the strong norm, hence we have to use another class of functions. We borrow the argument from  \cite{BMP,MPPP} and we argue using coarea formula: hence the claim is reduced to the claim for characteristic functions of sets of finite perimeter and finite measure. In this reduction procedure the regularizing property of the heat flow (thus, the curvature bound) heavily enters into play. Then the claim for this class of functions is proved (Theorem \ref{cdsnjca}) by algebraic manipulations and  Lemma \ref{maincomp}, which in turn builds upon the blow-up argument of Lemma \ref{lem:integr_limit}. Again, the blow-up argument relies on the convergence of the heat flow, which, as before, relies on the curvature bound. In this part, the main difficulty we have to face is the lack of linear structure of the space (the linear structure was present in \cite{BMP,MPPP}) which forces us to exploit a different procedure.

\section{Preliminaries}
\subsection{Sobolev calculus}
By a \emph{metric measure space} we mean a triplet \((\XX,\dist,\mass)\), where:
\[\begin{split}
(\XX,\dist),&\quad\text{ is a complete and separable metric space,}\\
\mass\geq 0,&\quad\text{ is a boundedly-finite Borel measure on }\XX.
\end{split}\]
Given any \(p\in(1,\infty)\), the \emph{Cheeger \(p\)-energy} \({\rm Ch}_p(f)\) of a function \(f\in L^p(\mass)\) is defined as
\[
{\rm Ch}_p(f)\coloneqq\inf_{(f_n)}\liminf_{n\to\infty}\int\lip(f_n)^p\,\dd\mass,
\]
where the infimum is taken among all sequences \((f_n)_n\) of boundedly-supported Lipschitz functions \(f_n\colon\XX\to\RR\)
such that \(f_n\to f\) in \(L^p(\mass)\). Here, the \emph{slope} \(\lip(f_n)\colon\XX\to[0,+\infty)\) is
\[
\lip(f_n)(x)\coloneqq\limsup_{y\to x}\frac{|f_n(x)-f_n(y)|}{\dist(x,y)},\quad\text{ for every accumulation point }x\in\XX,
\]
and \(\lip(f_n)(x)\coloneqq 0\) for every isolated point \(x\in\XX\). The \emph{Sobolev space} \(W^{1,p}(\XX)\) is defined as
\[
W^{1,p}(\XX)\coloneqq\big\{f\in L^p(\mass)\;\big|\;{\rm Ch}_p(f)<+\infty\big\}.
\]
Given any \(f\in W^{1,p}(\XX)\), there exists a unique function \(|\dd_p f|\in L^p(\mass)\), called the \emph{minimal \(p\)-relaxed slope}
of \(f\), such that the Cheeger \(p\)-energy of \(f\) can be represented as
\[
{\rm Ch}_p(f)=\int|\dd_p f|^p\,\dd\mass.
\]
See \cite{AmbrosioGigliSavare11,AmbrosioGigliSAvare11-3} after \cite{Cheeger00} for the above notions and results.
\medskip

We will focus our attention on the class of \emph{infinitesimally Hilbertian} metric measure spaces \cite{Gigli12}. These are
metric measure spaces whose Cheeger \(2\)-energy \({\rm Ch}\coloneqq{\rm Ch}_2\colon L^2(\mass)\to[0,+\infty]\) is quadratic.
Under this assumption, for any \(f\in W^{1,2}(\XX)\) we will write \(|\nabla f|\) instead of \(|\dd_2 f|\). Moreover,
every Sobolev function \(f\in W^{1,2}(\XX)\) admits a quasi-continuous representative, which is uniquely determined
up to \(\capa\)-a.e.\ equality. Here, \(\capa\) stands for the \emph{Sobolev \(2\)-capacity} on \((\XX,\dist,\mass)\),
which is the outer measure on \(\XX\) defined as
\[
\capa(E)\coloneqq\inf_f\int|f|^2+|\nabla f|^2\,\dd\mass,\quad\text{ for every }E\subseteq\XX,
\]
the infimum being taken among all \(f\in W^{1,2}(\XX)\) satisfying \(f\geq 1\) \(\mass\)-a.e.\ on some open neighbourhood of \(E\).
A function \(\bar f\colon\XX\to\RR\) is \emph{quasi-continuous} if for any \(\varepsilon>0\) there exists a set \(E_\varepsilon\subseteq\XX\)
with \(\capa(E_\varepsilon)<\varepsilon\) such that \(f|_{\XX\setminus E_\varepsilon}\) is continuous.
See e.g.\ \cite{debin2019quasicontinuous}. Recall that \(\mass\ll\capa\),
meaning that \(\mass(N)=0\) whenever \(N\subseteq\XX\) is a Borel set such that \(\capa(N)=0\).
\medskip

On infinitesimally Hilbertian spaces \((\XX,\dist,\mass)\), the carr\'{e} du champ operator, given by
\[
\nabla f\cdot\nabla g\coloneqq\frac{|\nabla(f+g)|^2-|\nabla f|^2-|\nabla g|^2}{2}\in L^1(\mass),\quad\text{ for every }f,g\in W^{1,2}(\XX),
\]
is bilinear. A function \(f\in W^{1,2}(\XX)\) has a \emph{Laplacian} if there exists \(\Delta f\in L^2(\mass)\) such that
\begin{equation}\label{eq:def_Laplacian}
\int\nabla f\cdot\nabla g\,\dd\mass=-\int g\Delta f\,\dd\mass,\quad\text{ for every }g\in W^{1,2}(\XX).
\end{equation}
Given that \(W^{1,2}(\XX)\) is dense in \(L^2(\mass)\), we have that \eqref{eq:def_Laplacian} uniquely determines \(\Delta f\).
We denote by \(D(\Delta)\) the space of all functions in \(W^{1,2}(\XX)\) having a Laplacian. It holds that \(D(\Delta)\)
is a vector subspace of \(W^{1,2}(\XX)\) and that \(\Delta\colon D(\Delta)\to L^2(\mass)\) is a linear operator. The \emph{heat flow}
\((\heat_t)_{t>0}\), which is defined as the gradient flow in the Hilbert space \(L^2(\mass)\) of the Cheeger energy \(\rm Ch\)
(whose existence and uniqueness are guaranteed by the Komura--Brezis theory), can be characterized as follows: for any \(f\in L^2(\mass)\),
the curve \((0,+\infty)\ni t\mapsto\heat_t f\in L^2(\mass)\) is locally absolutely continuous, satisfies \(\lim_{t\searrow 0}\heat_t f=f\)
with respect to the \(L^2(\mass)\)-topology, and
\[
\heat_t f\in D(\Delta),\quad\frac{\dd}{\dd t}\heat_t f=\Delta\heat_t f,\quad\text{ for every }t>0.
\]
In the following statement, we collect some useful properties of the heat flow:
\begin{prop}
Let \((\XX,\dist,\mass)\) be an infinitesimally Hilbertian metric measure space. Then:
\begin{itemize}
\item[\(\rm i)\)] \textsc{Weak maximum principle.} Given any \(f\in L^2(\mass)\cap L^\infty(\mass)\) and \(t>0\), it holds that
\begin{equation}\label{eq:maximum_principle}
\heat_t f\leq\|f\|_{L^\infty(\mass)},\quad\text{ in the }\capa\text{-a.e.\ sense,}
\end{equation}
where we keep the same notation to denote some quasi-continuous representative of \(\heat_t f\).
\item[\(\rm ii)\)] \textsc{\(\Delta\) and \(\heat_t\) commute.} Given any \(f\in D(\Delta)\) and \(t>0\), it holds that
\begin{equation}\label{eq:Delta_h_commute}
\heat_t\Delta f=\Delta\heat_t f,\quad\text{ in the }\mass\text{-a.e.\ sense.}
\end{equation}
\item[\(\rm iii)\)] \textsc{\(\heat_t\) is self-adjoint.} Given any \(f,g\in L^2(\mass)\) and \(t>0\), it holds that
\begin{equation}\label{eq:h_self-adj}
\int g\,\heat_t f\,\dd\mass=\int f\,\heat_t g\,\dd\mass.
\end{equation}
\end{itemize}
\end{prop}
\begin{proof}
Proofs of i) (with \(\capa\) replaced by \(\mass\) in \eqref{eq:maximum_principle}), of ii), and of iii) can be found in
\cite{AmbrosioGigliSavare11} (see also \cite{GP19}). The full statement of i) then follows by observing that if a quasi-continuous
function \(f\colon\XX\to\RR\) satisfies \(f\geq 0\) \(\mass\)-a.e., then it satisfies \(f\geq 0\) \(\capa\)-a.e.. 
\end{proof}
We will work in the subclass of metric measure spaces that are usually called `PI spaces':
\begin{defn}
Let \((\XX,\dist,\mass)\) be a metric measure space. Then:
\begin{itemize}
\item[\(\rm i)\)] We say that \((\XX,\dist,\mass)\) is \emph{uniformly locally doubling} if for any radius \(R>0\)
there exists a constant \(C_D=C_D(R)>0\) such that
\[
\mass(B_{2r}(x))\leq C_D\mass(B_r(x)),\quad\text{ for every }x\in\XX\text{ and }r\in(0,R).
\]
\item[\(\rm ii)\)] We say that \((\XX,\dist,\mass)\) supports a \emph{weak local \((1,1)\)-Poincar\'{e} inequality}
if there exist a constant \(\lambda\geq 1\) and, for any radius \(R>0\), a constant \(C_P=C_P(R)>0\), such that
\[
\fint_{B_r(x)}\bigg|f-\fint_{B_r(x)}f\,\dd\mass\bigg|\,\dd\mass\leq C_P r\fint_{B_{\lambda r}(x)}\lip(f)\,\dd\mass,
\quad\text{ for every }x\in\XX\text{ and }r\in(0,R),
\]
whenever \(f\colon\XX\to\RR\) is a compactly-supported Lipschitz function.
\item[\(\rm iii)\)] We say that \((\XX,\dist,\mass)\) is a \emph{PI space} if it is uniformly locally doubling and
it supports a weak local \((1,1)\)-Poincar\'{e} inequality.
\end{itemize}
\end{defn}
A key property of PI spaces, which follows from the results of \cite{Cheeger00}, is the following:
\begin{equation}\label{eq:|Df|=lipf}
(\XX,\dist,\mass)\;\text{ PI space}\quad\Longrightarrow\quad|\dd_p f|=\lip(f)\,\text{ for every }f\in\LIP_{\rm bs}(\XX)\text{ and }p\in(1,\infty).
\end{equation}
The following result is taken from \cite[Proposition 2.17]{GT21} (see also \cite{Hajasz:1996aa} and \cite{HKST15}).
\begin{prop}\label{prop:Hajlasz}
Let \((\XX,\dist,\mass)\) be a PI space and \(p\in(1,\infty)\). Let \(f\in W^{1,p}(\XX)\) be given. Then for any \(R>0\)
there exists a \emph{Haj\l asz upper gradient} \(g_R\) of \(f\) at scale \(R\), namely \(g_R\in L^p(\mass)\) is a non-negative
function such that, for some Borel set \(N\subseteq\XX\) with \(\mass(N)=0\), it holds that
\[
|f(x)-f(y)|\leq(g_R(x)+g_R(y))\dist(x,y),\quad\text{ for every }x,y\in\XX\setminus N\text{ with }\dist(x,y)\leq R.
\]
Moreover, the function \(g_R\) can be also chosen so that \(\|g_R\|_{L^p(\mass)}\leq C\|f\|_{W^{1,p}(\XX)}\), for some
constant \(C>0\) depending only on the doubling constant, the Poincar\'{e} constants, \(R\), and \(p\).
\end{prop}
\subsection{Functions of bounded variation}
Let \((\XX,\dist,\mass)\) be a metric measure space. Fix a function \(f\in L^1_{\rm loc}(\mass)\) and an open set \(\Omega\subseteq\XX\).
Following \cite{MIRANDA2003}, we define the \emph{total variation} of \(f\) on \(\Omega\) as
\[
|\DIFF f|(\Omega)\coloneqq\inf_{(f_n)}\liminf_{n\to\infty}\int_\Omega\lip(f_n)\,\dd\mass,
\]
where the infimum is taken among all sequences \((f_n)_n\) of locally Lipschitz functions \(f_n\colon\Omega\to\RR\) such that
\(f_n\to f\) in \(L^1_{\rm loc}(\mass\mres\Omega)\). When \(|\DIFF f|(\XX)<+\infty\), we have that the set-function \(|\DIFF f|\)
is the restriction of a finite Borel measure on \(\XX\), which we still denote by \(|\DIFF f|\). We define
\[
\BV(\XX)\coloneqq\big\{f\in L^1(\mass)\;\big|\;|\DIFF f|(\XX)<+\infty\big\}.
\]
The elements of \(\BV(\XX)\) are called \emph{functions of bounded variation}. Similarly, a Borel set \(E\subseteq\XX\) is said to be
of \emph{finite perimeter} if \(|\DIFF\nchi_E|(\XX)<+\infty\). A key property is the \emph{coarea formula}:
\begin{thm}[Coarea]
Let \((\XX,\dist,\mass)\) be a metric measure space. Let \(f\in L^1_{\rm loc}(\mass)\) be given. Then for any Borel set \(E\subseteq\XX\)
it holds that the function \(\RR\ni t\mapsto|\DIFF\nchi_{\{f>t\}}|(E)\in[0,+\infty]\) is Borel measurable and
\[
|\DIFF f|(E)=\int_\RR|\DIFF\nchi_{\{f>t\}}|(E)\,\dd t.
\]
In particular, it holds that \(f\in\BV(\XX)\) if and only if \(\{f>t\}\) is a set of finite perimeter for a.e.\ \(t\in\RR\)
and the function \(t\mapsto|\DIFF\nchi_{\{f>t\}}|(\XX)\) belongs to \(L^1(\RR)\).
\end{thm}
The following proposition contains a couple of results (proved in \cite{amb01}) concerning sets of finite perimeter
that are now well-known in the context of $\PI$ spaces; see also \cite{AmbrosioAhlfors}.
\begin{prop}\label{zuppa}
Let $(\XX,\dist,\mass)$ be a $\PI$ space and let $E\subseteq\XX$ be a set of locally finite perimeter. Then, for $|\DIFF\nchi_E|$-a.e.\ $x\in\XX$ the following hold:
\begin{enumerate}[label=\roman*)]
    \item we have the following estimates
\begin{equation}\label{zuppa1eq}
       0<\liminf_{r\searrow 0} \frac{r |\DIFF\nchi_E|(B_r(x))}{\mass(B_r(x))}\le \limsup_{r\searrow 0}\frac{r |\DIFF\nchi_E|(B_r(x))}{\mass(B_r(x))}<\infty,
\end{equation}    
    \item the following density estimate holds
    \begin{equation}\label{zuppa2eq}
     \liminf_{r\searrow 0} \min\bigg\{\frac{\mass( B_r(x)\cap E)}{\mass(B_r(x))},\frac{\mass( B_r(x)\setminus E)}{\mass(B_r(x))}\bigg\}>0.
    \end{equation}
\end{enumerate}
\end{prop}
\subsection{The theory of \texorpdfstring{\(\RCD\)}{RCD} spaces}
We assume the reader is already familiar with the language of \(\RCD(K,N)\) spaces, where \(K\in\RR\) and \(N\in[1,\infty]\).
These are infinitesimally Hilbertian metric measure spaces whose Ricci curvature is bounded below by \(K\) and whose dimension is
bounded above by \(N\), in a synthetic sense. See \cite{AmbICM} and the references therein for a detailed account of this topic.
\medskip

\subsubsection*{Pointed-measured-Gromov-Hausdorff convergence and tangent cones}
A fundamental concept in the \(\RCD\) theory is the \emph{pointed-measured-Gromov-Hausdorff} convergence (\emph{pmGH} convergence for short).
Let \(\big((\XX_n,\dist_n,\mass_n,x_n)\big)_n\) be pointed \(\RCD(K,N)\) spaces, for some \(K\in\RR\) and \(N\in[1,\infty)\).
Then we say that \((\XX_n,\dist_n,\mass_n,x_n)\) converges to some pointed metric measure space \((\XX_\infty,\dist_\infty,\mass_\infty,x_\infty)\)
in the pmGH sense provided there exist a proper metric space \((\ZZ,\dist_\ZZ)\) and isometric embeddings \(\iota_n\colon\XX_n\hookrightarrow\ZZ\) for
\(n\in\NN\cup\{\infty\}\) such that
\[
(\iota_n)_\#\mass_n\rightharpoonup(\iota_\infty)_\#\mass_\infty,\quad\text{ in duality with }C_{\rm bs}(\ZZ),
\]
and \(\iota_n(x_n)\to\iota_\infty(x_\infty)\). About pmGH, the class of \(\RCD\) spaces enjoys two key properties:
\begin{itemize}
\item[\(\rm i)\)] \textsc{Compactness.} Any given sequence \(\big((\XX_n,\dist_n,\mass_n,x_n)\big)_n\) of pointed \(\RCD(K,N)\) spaces
pmGH-subconverges to some limit pointed metric measure space \((\XX_\infty,\dist_\infty,\mass_\infty,x_\infty)\).
\item[\(\rm ii)\)] \textsc{Stability.} If each \((\XX_n,\dist_n,\mass_n)\) is an \(\RCD(K_n,N_n)\) space and \((K_n,N_n)\to(K_\infty,N_\infty)\),
then the pmGH limit \((\XX_\infty,\dist_\infty,\mass_\infty)\) is an \(\RCD(K_\infty,N_\infty)\) space.
\end{itemize}
We refer to \cite{GMS15} and the references therein for more details.
\medskip

Tangents to an \(\RCD\) space are typically intended in the pmGH sense. Given a pointed \(\RCD(K,N)\) space \((\XX,\dist,\mass,x)\)
and a radius \(r>0\), we define the \emph{normalized measure} \(\mass_r^x\) as
\[
\mass_r^x\coloneqq\frac{\mass}{\mass(B_r(x))}.
\]
The rescaled space \((\XX,r^{-1}\dist,\mass_r^x)\) is an \(\RCD(r^2K,N)\) space. The \emph{tangent cone} \({\rm Tan}_x(\XX,\dist,\mass)\)
is then defined as the family of all those pointed metric measure spaces \((\YY,\dist_\YY,\mass_\YY,y)\) such that
\((\XX,r_i^{-1}\dist,\mass_{r_i}^x,x)\to(\YY,\dist_\YY,\mass_\YY,y)\) in the pmGH sense for some \((r_i)_i\) with \(r_i\searrow 0\).
Thanks to the compactness and stability properties of the pmGH convergence, we have \({\rm Tan}_x(\XX,\dist,\mass)\neq\varnothing\)
for every \(x\in\XX\) and that all its elements are pointed \(\RCD(0,N)\) spaces.
\subsubsection*{Structure theory of \(\RCD\) spaces}
We collect in the following statement several results \cite{Gigli_2015,Mondino-Naber14,KelMon16,DPMR16,GP16-2,bru2018constancy}
concerning the structure of finite-dimensional \(\RCD\) spaces. By \(\mathcal H^n\) we will mean the \(n\)-dimensional Hausdorff
measure on \((\XX,\dist)\).
\begin{thm}
Let \((\XX,\dist,\mass)\) be an \(\RCD(K,N)\) space, for some \(K\in\RR\) and \(N\in[1,\infty)\). Then there exists
a unique \(n\in\NN\) with \(n\leq N\), called the \emph{essential dimension} of \((\XX,\dist,\mass)\), such that
\[
\mass(\XX\setminus\mathcal R_n)=0,\quad\text{ where }
\mathcal R_n\coloneqq\Big\{x\in\XX\;\Big|\;{\rm Tan}_x(\XX,\dist,\mass)=\big\{(\RR^n,\dist_e,\tilde{\mathcal L}^n,0)\big\}\Big\},
\]
while \(\dist_e\) stands for the Euclidean distance in \(\RR^n\) and \(\tilde{\mathcal L}^n\coloneqq(\mathcal L^n)_1^0\).
Moreover, \((\XX,\dist)\) is \((\mass,n)\)-rectifiable (i.e.\ it can be covered, up to \(\mass\)-null sets, by countably
many Borel sets that are biLipschitz equivalent to subsets of \(\RR^n\)) and \(\mass=\theta\mathcal H^n\) for some
\(\theta\colon\XX\to(0,+\infty)\) Borel.
\end{thm}

In the present paper, we are mostly concerned with functions of bounded variation and sets of finite perimeter in the setting of
\(\RCD\) spaces. In this regard, an important concept is the notion of \emph{tangent to a set of finite perimeter}, which has
been introduced in \cite{ambrosio2018rigidity}:
\begin{defn}\label{def:pmGH}
Let \((\XX,\dist,\mass)\) be an \(\RCD(K,N)\) space, for some \(K\in\RR\) and \(N\in[1,\infty)\). Let \(E\subseteq\XX\) be a set
of locally finite perimeter. Then for any \(x\in\XX\) we denote by \({\rm Tan}_x(\XX,\dist,\mass,E)\) the family of all quintuplets
\((\YY,\dist_\YY,\mass_\YY,y,F)\) such that the following conditions are verified:
\begin{itemize}
\item[\(\rm i)\)] \((\YY,\dist_\YY,\mass_\YY,y)\in{\rm Tan}_x(\XX,\dist,\mass)\).
\item[\(\rm ii)\)] \(F\subseteq\YY\) is a set of locally finite perimeter with \(\mass_\YY(F)>0\).
\item[\(\rm iii)\)] For some sequence \(r_i\searrow 0\) and some proper metric space \((\ZZ,\dist_\ZZ)\) where the pmGH convergence
\((\XX,r_i^{-1}\dist,\mass_{r_i}^x,x)\to(\YY,\dist_\YY,\mass_\YY,y)\) is realized, it holds that for every $R>0$
\[
\nchi_{E\cap B_{Rr_i}(x)}\mass_{r_i}^x\rightharpoonup\nchi_{F\cap B_R(y)}\mass_\YY,\quad\text{ in duality with }C_{\rm bs}(\ZZ).
\]
\end{itemize}
\end{defn}

The following result, which is the outcome of several contributions \cite{ambrosio2018rigidity,bru2019rectifiability,bru2021constancy,ABP22},
can be regarded as the generalization to the \(\RCD\) setting of De Giorgi's structure theorem for sets of finite perimeter in the Euclidean space.
By \(\HH^h\) we mean the codimension-one Hausdorff measure on \((\XX,\dist,\mass)\), obtained via Carath\'{e}odory's construction
with \(h(B_r(x))\coloneqq\frac{\mass(B_r(x))}{2r}\) as a gauge function. Namely, given any set \(E\subseteq\XX\), we define
\begin{equation}\label{eq:cod_1_Hausd}
\HH^h(E)\coloneqq\sup_{\delta>0}\,\inf\sum_{i\in\NN}\frac{\mass(B_{r_i}(x_i))}{2r_i},
\end{equation}
the infimum being among all sequences of balls \((B_{r_i}(x_i))_{i\in\NN}\)
with \(r_i<\delta\) and \(E\subseteq\bigcup_{i\in\NN}B_{r_i}(x_i)\).
\begin{thm}
Let \((\XX,\dist,\mass)\) be an \(\RCD(K,N)\) space, for some \(K\in\RR\) and \(N\in[1,\infty)\), of essential dimension \(n\).
Let \(E\subseteq\XX\) be a given set of finite perimeter. Then the perimeter measure \(|\DIFF\nchi_E|\) is concentrated on the \emph{reduced boundary}
\(\mathcal F E\) of \(E\), which is defined as
\[
\mathcal F E\coloneqq\Big\{x\in\XX\;\Big|\;{\rm Tan}_x(\XX,\dist,\mass,E)=\big\{(\RR^n,\dist_e,\tilde{\mathcal L}^n,0,\{z_n>0\})\big\}\Big\}.
\]
Moreover, \((\XX,\dist)\) is \((|\DIFF\nchi_E|,n-1)\)-rectifiable and \(|\DIFF\nchi_E|=\Theta_n(\mass,\cdot)\mathcal H^{n-1}\mres{\mathcal F E}\),
where \(\Theta_n(\mass,\cdot)\) is given by \(\Theta_n(\mass,x)=\lim_{r\searrow 0}\frac{\mass(B_r(x))}{\omega_n r^n}\) for
\(|\DIFF\nchi_E|\)-a.e.\ \(x\in\XX\), with \(\omega_n\coloneqq\mathcal L^n(B_1(0))\).
Also,
\begin{equation}\label{reprformula}
    |\DIFF\nchi_E|=\frac{\omega_{n-1}}{\omega_n}\HH^h\mres\FF E.
\end{equation}
\end{thm}
\subsubsection*{Heat flow on \(\RCD\) spaces}
The \(\RCD\) condition entails good properties at the level of the heat flow semigroup. A first important feature is the
\emph{\(L^\infty\)-Lipschitz regularization property}: given any function \(f\in L^\infty(\mass)\) in an \(\RCD(K,\infty)\) space and \(t\in(0,1]\),
it holds \(|\nabla\heat_t f|\in L^\infty(\mass)\) and
\begin{equation}\label{eq:Linfty_to_Lip}
\||\nabla\heat_t f|\|_{L^\infty(\mass)}\leq\frac{\|f\|_{L^\infty(\mass)}}{e^{K^-}\sqrt{2t}}.
\end{equation}
In particular, \(\heat_t f\) admits a Lipschitz representative (which we still denote by \(\heat_t f\)) having Lipschitz constant
\(\frac{\|f\|_{L^\infty(\mass)}}{e^{K^-}\sqrt{2t}}\); this is a consequence of the so-called \emph{Sobolev-to-Lipschitz property}
of \(\RCD\) spaces, which states that every Sobolev function \(f\in W^{1,2}(\XX)\) satisfying \(|\nabla f|\leq 1\)
\(\mass\)-a.e.\ has a \(1\)-Lipschitz representative. Notice also that \eqref{eq:maximum_principle} ensures that for any \(f\in L^\infty(\mass)\)
\begin{equation}\label{eq:maximum_principle_RCD}
\heat_t f\leq\|f\|_{L^\infty(\mass)},\quad\text{ everywhere on }\XX.
\end{equation}
Consequently, given any \(\mu\in\mathscr P_2(\XX)\)  (i.e.\ \(\mu\)
is a Borel probability measure on \(\XX\) with finite second-moment, in the sense that \(\int\dist(\cdot,\bar x)^2\,\dd\mu<+\infty\)
for every \(\bar x\in\XX\)), it makes sense to define \(\mathscr H_t\mu\) for any \(t>0\) as the unique element of \(\mathscr P_2(\XX)\)
satisfying the identity
\begin{equation}\label{eq:def_dual_heat_flow}
\int f\,\dd(\mathscr H_t\mu)=\int\heat_t f\,\dd\mu,\quad\text{ for every }f\in C_{\rm bs}(\XX).
\end{equation}
It holds that \(\mathscr H_t\mu\ll\mass\) for every \(t>0\). We denote by \(\heat_t^*\mu\in L^1(\mass)\) the Radon--Nikod\'{y}m
derivative \(\frac{\dd(\mathscr H_t\mu)}{\dd\mass}\). Then we define the \emph{heat kernel}
\((0,+\infty)\times\XX\times\XX\ni(t,x,y)\mapsto p_t(x,y)\) as
\[
p_t(x,\cdot)\coloneqq\heat_t^*\delta_x,\quad\text{ where }\delta_x\text{ stands for the Dirac measure at }x.
\]
We will often use the \emph{\(1\)-Bakry--\'{E}mery contraction estimate}, which (in this generality) was proved in \cite{BG22} building upon \cite{Savare13} (see also \cite{GigliHan14}):
given \(f\in\BV(\XX)\cap L^\infty(\mass)\) and \(t>0\), it holds that \(\heat_t f\in\BV(\XX)\cap W^{1,2}(\XX)\) and
\begin{equation}\label{eq:Bakry-Emery}
|\nabla\heat_t f|\leq e^{-Kt}\heat_t^*|\DIFF f|,\quad\text{ in the }\mass\text{-a.e.\ sense.}
\end{equation}
We now focus on \(\RCD(K,N)\) spaces \((\XX,\dist,\mass)\) with \(N<\infty\). These spaces
\((\XX,\dist,\mass)\) are uniformly locally doubling \cite{Lott-Villani09,Sturm06II} and support weak local Poincar\'{e}
inequalities \cite{Rajala12} (i.e.\ they are PI), thus accordingly \((t,x,y)\mapsto p_t(x,y)\) admits a locally H\"{o}lder continuous
representative by \cite{Sturm96II,Sturm96III}. We collect in the next statement other properties of \(p_t(\cdot,\cdot)\):
\begin{prop}\label{prop:properties_heat_kernel}
Let \((\XX,\dist,\mass)\) be an \(\RCD(K,N)\) space, with \(K\in\RR\) and \(N\in[1,\infty)\). Then:
\begin{itemize}
\item[\(\rm i)\)] There exist constants \(C_1=C_1(K,N)>0\) and \(C_2=C_2(K,N)>0\) such that
\begin{equation}\label{eq:heat_kernel_ptwse_est}
\frac{1}{C_1\mass(B_{\sqrt t}(x))}\exp\bigg\{-\frac{\dist(x,y)^2}{3t}-C_2 t\bigg\}\leq p_t(x,y)\leq
\frac{C_1}{\mass(B_{\sqrt t}(x))}\exp\bigg\{-\frac{\dist(x,y)^2}{5t}+C_2 t\bigg\}
\end{equation}
holds for every \(x,y\in\XX\) and \(t>0\).
\item[\(\rm ii)\)] There exists a constant \(C_3=C_3(K,N)>0\) such that for any \(\alpha>1\) it holds that
\begin{equation}\label{eq:heat_kernel_int_est}
\int_{\XX\setminus B_{\alpha r}(x)}p_{r^2}(x,y)\,\dd\mass(y)\leq C_3 e^{-\frac{\alpha^2}{24}},\quad\text{ for every }x\in\XX\text{ and }r>0.
\end{equation}
\end{itemize}
\end{prop}
\begin{proof}
The Gaussian estimates in i) were obtained in \cite{jiang2014heat}. 
Item ii) is an easy consequence of the Gaussian estimates of item i) and of the doubling property of the measure;
we refer to \cite[proof of Lemma 3.23, \textsc{Step 1}]{BG22} for a proof.
\end{proof}

Thanks to Proposition \ref{prop:properties_heat_kernel} i), for any finite Borel measure \(\mu\geq 0\) on \(\XX\) we can define
\begin{equation}\label{eq:alt_def_dual_heat_flow}
\heat_t^*\mu(x)\coloneqq\int p_t(x,y)\,\dd\mu(y),\quad\text{ for }\mass\text{-a.e.\ }x\in\XX.
\end{equation}
Using Fubini's theorem, one can check that this definition is consistent with the one in \eqref{eq:def_dual_heat_flow} when
\(\mu\in\mathscr P_2(\XX)\), and that \(\heat_t^*(f\mass)=\heat_t f\) for every \(f\in L^1(\mass)^+\). Moreover, we have that
\begin{equation}\label{eq:dual_mass_preserv}
\int\heat_t^*\mu\,\dd\mass=\mu(\XX),\quad\text{ for every finite Borel measure }\mu\geq 0\text{ on }\XX.
\end{equation}
\subsubsection*{Vector fields and Gauss--Green formula}
To introduce the notion of polar vector field of a \(\BV\) function in the \(\RCD\) setting, we first need to recall the vector calculus
for metric measure spaces developed by Gigli in \cite{Gigli14}. The relevant concept is that of an \emph{\(L^p\)-normed \(L^\infty\)-module}.
For the purposes of this paper, the following definitions (taken in this formulation from \cite[Section 4.2]{GLP22}) are sufficient.
\medskip

Let \((\XX,\Sigma,\mathcal N)\) be an \emph{enhanced} measurable space, i.e.\ a measurable space \((\XX,\Sigma)\) together with a \(\sigma\)-ideal
\(\mathscr N\subseteq\Sigma\) (namely, a family of measurable sets containing the empty set, which is closed under taking subsets and countable unions).
A significant example is the \(\sigma\)-ideal \(\mathcal N_\mu\coloneqq\{N\in\Sigma\,:\,\mu(N)=0\}\) induced by an outer measure \(\mu\) on \(\XX\).
We denote by \(\mathcal L^\infty(\Sigma)\) the vector space of all bounded measurable functions from \((\XX,\Sigma)\) to \(\RR\). The \(\sigma\)-ideal
\(\mathcal N\) induces an equivalence relation \(\sim_{\mathcal N}\) on \(\mathcal L^\infty(\Sigma)\) as follows: given any
\(f,g\in\mathcal L^\infty(\Sigma)\), we declare that \(f\sim_{\mathcal N}g\) provided \(f=g\) \emph{holds \(\mathcal N\)-a.e.}, meaning
that \(\{x\in\XX\,:\,f(x)\neq g(x)\}\in\mathcal N\). We denote by \(L^\infty(\XX,\mathcal N)\) the quotient space
\(\mathcal L^\infty(\Sigma)/\sim_{\mathcal N}\) and by \(\pi_{\mathcal N}\colon\mathcal L^\infty(\Sigma)\to L^\infty(\XX,\mathcal N)\)
the canonical projection operator. Notice that \(L^\infty(\XX,\mathcal N)\) is a vector space and a commutative ring with respect to
the natural pointwise operations. It is also a Banach space if endowed with
\[
\|f\|_{L^\infty(\XX,\mathcal N)}\coloneqq\inf_{N\in\mathcal N}\sup_{\XX\setminus N}|\bar f|,
\quad\text{ for every }f=\pi_{\mathcal N}(\bar f)\in L^\infty(\XX,\mathcal N).
\]
In the case where \(\mathcal N=\mathcal N_\mu\) for some (outer) measure \(\mu\) on \(\XX\), we write \(L^\infty(\mu)\coloneqq L^\infty(\XX,\mathcal N_\mu)\).
\begin{defn}
Let \((\XX,\Sigma,\mathcal N)\) be an enhanced measurable space. Then an algebraic module \(\mathscr M\) over \(L^\infty(\XX,\mathcal N)\)
is said to be an \emph{\(L^\infty(\XX,\mathcal N)\)-normed \(L^\infty(\XX,\mathcal N)\)-module} provided it is endowed with a
\emph{pointwise norm} operator \(|\cdot|\colon\mathscr M\to L^\infty(\XX,\mathcal N)\) verifying the following properties:
\begin{itemize}
\item[\(\rm i)\)] Given any \(v,w\in\mathscr M\) and \(f\in L^\infty(\XX,\mathcal N)\), it \(\mathcal N\)-a.e.\ holds that
\[\begin{split}
|v|&\geq 0,\quad\text{ with equality if and only if }v=0,\\
|v+w|&\leq|v|+|w|,\\
|fv|&=|f||v|.
\end{split}\]
\item[\(\rm ii)\)] For any partition \((E_n)_{n\in\NN}\subseteq\Sigma\) of \(\XX\) and \((v_n)_{n\in\NN}\subseteq\mathscr M\)
with \(\sup_n\|\nchi_{E_n}|v_n|\|_{L^\infty(\XX,\mathcal N)}<\infty\), there exists an element \(v\in\mathscr M\) such that
\(\nchi_{E_n}v=\nchi_{E_n}v_n\) for every \(n\in\NN\).
\item[\(\rm iii)\)] The norm \(\|v\|_{\mathscr M}\coloneqq\||v|\|_{L^\infty(\XX,\mathcal N)}\) on \(\mathscr M\) is complete.
\end{itemize}
\end{defn}
The element \(v\) in item ii) is uniquely determined, thus it can be unambiguously denoted by \(\sum_{n\in\NN}\nchi_{E_n}v_n\),
even though it is not necessarily any limit of the finite sums \(\sum_{n=1}^N\nchi_{E_n}v_n\).
\medskip

Next we recall the concept of \emph{capacitary tangent module} on an \(\RCD\) space. Before doing it, we remind
that (following \cite{Savare13}) the \emph{test functions} on an \(\RCD(K,\infty)\) space \((\XX,\dist,\mass)\) are
\[
{\rm Test}(\XX)\coloneqq\big\{f\in D(\Delta)\cap{\rm LIP}_{\rm b}(\XX)\;\big|\;\Delta f\in W^{1,2}(\XX)\big\},
\]
that \({\rm Test}(\XX)\) is dense in \(W^{1,2}(\XX)\), and that \(\nabla f\cdot\nabla g\in W^{1,2}(\XX)\) holds whenever \(f,g\in{\rm Test}(\XX)\).
In particular, \(|\nabla f|\in W^{1,2}(\XX)\) as a consequence of Kato's inequality \cite[Lemma 3.5]{debin2019quasicontinuous}.
We can now state the following result, obtained by slightly adapting \cite[Theorem 3.6]{debin2019quasicontinuous}.
\begin{thm}
Let \((\XX,\dist,\mass)\) be an \(\RCD(K,\infty)\) space, for some \(K\in\RR\). Then there exists a unique couple
\((L^\infty_\capa(T\XX),\bar\nabla)\) having the following properties:
\begin{itemize}
\item[\(\rm i)\)] \(L^\infty_\capa(T\XX)\) is an \(L^\infty(\capa)\)-normed \(L^\infty(\capa)\)-module and
\(\bar\nabla\colon{\rm Test}(\XX)\to L^\infty_\capa(T\XX)\) is a linear operator.
\item[\(\rm ii)\)] It holds \(|\bar\nabla f|=\overline{|\nabla f|}\) in the \(\capa\)-a.e.\ sense for all \(f\in{\rm Test}(\XX)\),
where \(\overline{|\nabla f|}\in L^\infty(\capa)\) stands for the quasi-continuous representative of \(|\nabla f|\in W^{1,2}(\XX)\).
\item[\(\rm iii)\)] The family of all those elements of \(L^\infty_\capa(T\XX)\) that can be written as \(\sum_{n\in\NN}\nchi_{E_n}\bar\nabla f_n\)
for suitable \((f_n)_{n\in\NN}\subseteq{\rm Test}(\XX)\) are dense in \(L^\infty_\capa(T\XX)\).
\end{itemize}
Uniqueness is in the following sense: if another couple \((\mathscr M,\bar\nabla')\) has the same properties, then there
exists a unique isomorphism \(\Phi\colon L^\infty_\capa(T\XX)\to\mathscr M\) such that \(\Phi\circ\bar\nabla=\bar\nabla'\).
\end{thm}

Given any Borel measure \(\mu\geq 0\) on \(\XX\) with \(\mu\ll\capa\), it is possible to ``quotient \(L^\infty_\capa(T\XX)\) up to
\(\mu\)-a.e.\ equality''. More precisely, the natural projection map \(\pi_\mu\colon L^\infty(\capa)\to L^\infty(\mu)\) induces an equivalence
relation \(\sim_\mu\) on \(L^\infty_\capa(T\XX)\) as follows: given any \(v,w\in L^\infty_\capa(T\XX)\), we declare that \(v\sim_\mu w\)
if and only if \(\pi_\mu(|v-w|)=0\) holds \(\mu\)-a.e. Then the quotient space \(L^\infty_\mu(T\XX)\coloneqq L^\infty_\capa(T\XX)/\sim_\mu\)
inherits an \(L^\infty(\mu)\)-normed \(L^\infty(\mu)\)-module structure. We keep the same notation
\(\pi_\mu\colon L^\infty_\capa(T\XX)\to L^\infty_\mu(T\XX)\) to denote the canonical projection operator.
\medskip

When we consider the reference measure \(\mass\ll\capa\), we recover the notion of tangent module \(L^\infty(T\XX)\coloneqq L^\infty_\mass(T\XX)\)
on \((\XX,\dist,\mass)\) introduced by \cite{Gigli14}. Moreover, we recall that
\begin{equation}\label{eq:ac_Cap}
|\DIFF f|\ll\capa,\quad\text{ for every }f\in\BV(\XX).
\end{equation}
The proof of this fact was obtained in \cite{BG22} for general metric measure spaces, whereas in the particular setting of $\PI$ spaces
(in particular, for finite-dimensional $\RCD$ spaces) a different proof was previously obtained in \cite{bru2019rectifiability}.
Hence, the following statement is meaningful:
\begin{thm}\label{Gauss-Green}
Let \((\XX,\dist,\mass)\) be an \(\RCD(K,\infty)\) space, for some \(K\in\RR\). Let \(f\in\BV(\XX)\).
Then there exists a unique \(\nu_f\in L^\infty_f(T\XX)\coloneqq L^\infty_{|\DIFF f|}(T\XX)\) such that
\(|\nu_f|=1\) holds \(|\DIFF f|\)-a.e.\ and
\[
\int f\Delta g\,\dd\mass=-\int\nu_f\cdot\pi_{|\DIFF f|}(\bar\nabla g)\,\dd|\DIFF f|,\quad\text{ for every }g\in{\rm Test}(\XX).
\]
\end{thm}

This Gauss--Green integration-by-parts formula was obtained in \cite{BG22} after \cite{bru2019rectifiability}. See also \cite{BGlvm} for a different proof.
In the case where \(f=\nchi_E\) for some set \(E\subseteq\XX\) of finite perimeter, we write \(\nu_E\) instead of \(\nu_{\nchi_E}\).
\section{Main part}
\subsection{Auxiliary results}
We now prove some results that will lead to the proof of our main Theorems \ref{main1} and \ref{main2}.
\begin{lem}\label{membership}
Let \((\XX,\dist,\mass)\) be an \(\RCD(K,N)\) space, for some \(K\in\RR\) and \(N\in[1,\infty)\). Then:
\begin{itemize}
\item[$\rm i)$] Given any $p\in(1,\infty)$ and $f\in L^p(\mass)$, it holds that
\[
\varliminf_{t\searrow 0}\frac{1}{t^{p/2}}\int\!\!\!\int p_t(x,y)|f(x)-f(y)|^p\,\dd\mass(x)\,\dd\mass(y)<+\infty
\quad\Longrightarrow\quad f\in W^{1,p}(\XX).
\]
\item[$\rm ii)$] Given any $f\in L^1(\mass)$, it holds that
\[
\varliminf_{t\searrow 0}\frac{1}{\sqrt t}\int\!\!\!\int p_t(x,y)|f(x)-f(y)|\,\dd\mass(x)\,\dd\mass(y)<+\infty
\quad\Longrightarrow\quad f\in\BV(\XX).
\]
\end{itemize}
\end{lem}
\begin{proof}
Given any radius \(r>0\), we define the near-diagonal set \(\Delta_r\subseteq\XX\times\XX\) as
\[
\Delta_r\coloneqq\big\{(x,y)\in\XX\times\XX\;\big|\;\dist(x,y)<r\big\}.
\]
\(\rm i)\) Consider for any \(p\in(1,\infty)\) the Korevaar--Schoen energy (for real-valued functions) in the sense of
Gigli--Tyulenev \cite{GT21}, namely for any function \(f\in L^p(\mass)\) we define
\[
{\sf E}_p(f)\coloneqq\varliminf_{r\searrow 0}\int\!\!\!\fint_{B_r(x)}\frac{|f(x)-f(y)|^p}{r^p}\,\dd\mass(y)\,\dd\mass(x).
\]
Observe that \(\int\!\!\fint_{B_r(x)}|f(x)-f(y)|^p\,\dd\mass(y)\,\dd\mass(x)=\int_{\Delta_r}\frac{|f(x)-f(y)|^p}{\mass(B_r(x))}\,\dd(\mass\otimes\mass)(x,y)\), so that
\[
{\sf E}_p(f)=\varliminf_{r\searrow 0}\frac{1}{r^p}\int_{\Delta_r}\frac{|f(x)-f(y)|^p}{\mass(B_r(x))}\,\dd(\mass\otimes\mass)(x,y).
\]
We thus know from \cite[Corollary 3.10]{GT21} that for any \(f\in L^p(\mass)\) it holds that
\begin{equation}\label{GigTyu}
\varliminf_{r\searrow 0}\frac{1}{r^p}\int_{\Delta_r}\frac{|f(x)-f(y)|^p}{\mass(B_r(x))}\,\dd(\mass\otimes\mass)(x,y)<+\infty
\quad\Longrightarrow\quad f\in W^{1,p}(\XX).
\end{equation}
The lower bound in \eqref{eq:heat_kernel_ptwse_est} implies that for any \(t>0\) and \(x,y\in\XX\) it holds that
\[
\frac{\nchi_{\Delta_{\sqrt t}}(x,y)}{\mass(B_{\sqrt t}(x))}\leq C_1 e^{C_2 t+\frac{1}{3}}p_t(x,y).
\]
Therefore, i) follows from \eqref{GigTyu}.\\
{\rm ii)} Since \((\XX,\dist,\mass)\) is a PI space, we know from \cite[Theorem 3.1]{MMS16} that for any \(f\in L^1(\mass)\)
\begin{equation}\label{MarMirShan}
\varliminf_{r\searrow 0}\frac{1}{r}\int_{\Delta_r}\frac{|f(x)-f(y)|}{\sqrt{\mass(B_r(x))}\sqrt{\mass(B_r(y)}}\,\dd(\mass\otimes\mass)(x,y)<+\infty
\quad\Longrightarrow\quad f\in\BV(\XX).
\end{equation}
Notice also that the lower bound in \eqref{eq:heat_kernel_ptwse_est} implies that for any \(t>0\) and \(x,y\in\XX\) it holds
\[
\frac{\nchi_{\Delta_{\sqrt t}}(x,y)}{\sqrt{\mass(B_{\sqrt t}(x))}\sqrt{\mass(B_{\sqrt t}(y))}}\leq C_1 e^{C_2 t+\frac{1}{3}}p_t(x,y).
\]
Therefore, ii) follows from \eqref{MarMirShan}.
\end{proof}
Next we treat separately the cases \(p\in(1,\infty)\) and \(p=1\). 
We start with the case $p\in (1,\infty)$. The two following lemmas are inspired respectively by \cite[Lemma 3.4 and Proposition 3.3]{Gor}.
The proof is similar, but some care has to be taken due to the change of the sequence of mollifiers involved.
\begin{lem}\label{Gorny1}
Let $(\XX,\dist,\mass)$ be an $\RCD(K,N)$ space and let $p\in(1,\infty)$. Let $f\in { W}^{1,p}(\XX)$ be given.
Then for any $t\in (0,1)$ it holds that
\begin{equation*}
    \frac{1}{t^{p/2}}\int\!\!\!\int p_t(x,y)|f(x)-f(y)|^p\,\dd\mass(x)\,\dd\mass(y)\le C\|f\|_{W^{1,p}(\XX)}^p,
\end{equation*}
where $C>0$ is a constant that depends only on \(K\), \(N\), \(p\).
\end{lem}
\begin{proof}
Using twice the Gaussian estimates of Proposition \ref{prop:properties_heat_kernel}, the doubling property of the measure,
and the boundedness on $[0,+\infty)$ of the function $\alpha\mapsto e^{-7\alpha^2/60}\alpha^p$, we obtain that
\begin{equation}\label{eqcomoda}
    \begin{split}
p_{t^2}(x,y)\frac{\dist(x,y)^p}{t^p}&\le C\frac{1}{\mass(B_t(x))} e^{-\frac{\dist(x,y)^2}{5 t^2}}\frac{\dist(x,y)^p}{t^p}\\&= C\frac{\mass(B_{2t}(x))}{\mass(B_t(x))}\bigg(\frac{1}{\mass(B_{2t}(x))}e^{-\frac{\dist(x,y)^2}{3(2t)^2}}\bigg)
\bigg(e^{-\frac{7\dist(x,y)^2}{60 t^2}}\frac{\dist(x,y)^p}{t^p}\bigg)\\&\le C p_{4 t^2}(x,y),
    \end{split}
\end{equation}
where the constant $C>0$ may vary from line to line and depends exclusively on $K$, $N$, $p$.

Thanks to Proposition \ref{prop:Hajlasz}, there exists a Haj\l asz upper gradient $g\in L^p(\mass)$ of the function $f$ at scale $1$
such that \(\|g\|_{L^p(\mass)}\leq C\|f\|_{W^{1,p}(\XX)}\). Then
\begin{align*}
    &\frac{1}{t^{p}}\int_{\{(x,y)\in\XX^2\,:\,\dist(x,y)\le 1\}} p_{t^2}(x,y)|f(x)-f(y)|^p\,\dd(\mass\otimes\mass)(x,y)\\
    \overset{\phantom{\eqref{eqcomoda}}}=\,&\int_{\{(x,y)\in\XX^2\,:\,\dist(x,y)\le 1\}} p_{t^2}(x,y)\frac{\dist(x,y)^p}{t^p}\frac{|f(x)-f(y)|^p}{\dist(x,y)^p}\,\dd(\mass\otimes\mass)(x,y)\\
    \overset{\eqref{eqcomoda}}\le\,& C\int\!\!\!\int p_{4 t^2}(x,y) (g(x)^p+g(y)^p)\,\dd\mass(x)\,\dd\mass(y)\\
    \overset{\phantom{\eqref{eqcomoda}}}=\,& 2C\int\heat_{4t^2}(g^p)\,\dd\mass\leq
    C\Vert g\Vert_{L^p(\mass)}^p\le C\Vert f\Vert_{W^{1,p}(\XX)}^p.
\end{align*}
Also, we can estimate
\begin{align*}
    &\frac{1}{t^{p}}\int_{\{(x,y)\in\XX^2\,:\,\dist(x,y)>1\}}p_{t^2}(x,y)|f(x)-f(y)|^p\,\dd(\mass\otimes\mass)(x,y)\\
    \overset{\phantom{\eqref{eqcomoda}}}\leq\,&\int_{\{(x,y)\in\XX^2\,:\,\dist(x,y)>1\}}
    p_{t^2}(x,y)\frac{\dist(x,y)^p}{t^p}{|f(x)-f(y)|^p}\,\dd(\mass\otimes\mass)(x,y)\\
    \overset{\eqref{eqcomoda}}\le\,& C \int\!\!\!\int p_{4 t^2}(x,y)(|f(x)|^p+|f(y)|^p)\,\dd\mass(x)\,\dd\mass(y)\\
    \overset{\phantom{\eqref{eqcomoda}}}=\,& 2C\int\heat_{4t^2}(|f|^p)\,\dd\mass\leq C\Vert f\Vert_{L^p(\mass)}^p
    \le C\Vert f\Vert_{W^{1,p}(\XX)}^p,
\end{align*}
and the combination of the two chains of inequalities above concludes the proof.
\end{proof}
\begin{lem}\label{Gorny2}
Let $(\XX,\dist,\mass)$ be an $\RCD(K,N)$ space, \(p\in(1,\infty)\), and $f\in\LIP_{\rm bs}(\XX)$. Then
\begin{equation*}
\frac{1}{t^{p/2}}\int p_t (x,y)|f(y)-f(x)|^p\,\dd\mass(y)=\frac{2^p}{\sqrt{\pi}}\Gamma\bigg(\frac{p+1}{2}\bigg)\lip (f)(x)^p,
\quad\text{ for every }x\in\mathcal R_n. 
\end{equation*}
\end{lem}
\begin{proof}
Fix $x\in\mathcal R_n$. We can assume with no loss of generality that $f(x)=0$, so that we have to show that given any sequence $\{t_i\}_i$
with $t_i\searrow 0$ we have, up to not relabeled subsequences,
 \begin{equation*}
\int p_{t_i^2} (x,y)\bigg(\frac{|f(y)|}{t_i}\bigg)^p\,\dd\mass(y)\to\frac{2^p}{\sqrt{\pi}}\Gamma\bigg(\frac{p+1}{2}\bigg) \lip (f)(x)^p.
\end{equation*}
Up to a not relabeled subsequence, we have
that \((\XX,t_i^{-1}\dist,\mass_{t_i}^x,x,E)\to(\RR^n,\dist_e,\tilde{\mathcal L}^n,0,H)\) as \(i\to\infty\) in the sense
of Definition \ref{def:pmGH}, where we set \(H\coloneqq\big\{(z_1,\ldots,z_n)\in\RR^n\,:\,z_n>0\big\}\), with
\(n\in\NN\) being the essential dimension of \((\XX,\dist,\mass)\). This follows as in the first part of the proof of \cite[Corollary 4.10]{ambrosio2018rigidity}, exploiting \cite[Corollary 3.4]{ambrosio2018rigidity}.
Let \((\ZZ,\dist_\ZZ)\) be a proper metric space
where \((\XX,t_i^{-1}\dist,\mass_{t_i}^x,x,E)\to(\RR^n,\dist_e,\tilde{\mathcal L}^n,0,H)\) is realized.
We denote by \(\#^i\) and \(\tilde\#\) the various objects in the rescaled space \((\XX,t_i^{-1}\dist,\mass_{t_i}^x,x)\)
and in the limit space \((\RR^n,\dist_e,\tilde{\mathcal L}^n,0)\), respectively. Recall that \(p_1^i(x,y)=p_{t_i^2}(x,y)\mass(B_{t_i}(x))\).
Hence we have to prove that 
 \begin{equation*}
\int p^{i}_1 (x,y)|f_i(y)|^p\,\dd\mass_{t_i}^x(y)\to\frac{2^p}{\sqrt{\pi}}\Gamma\bigg(\frac{p+1}{2}\bigg) \lip (f)(x)^p,
\end{equation*}
where $f_i=f/t_i$ in the rescaled space. 

Let $R>0$ and notice that, using \eqref{eqcomoda} in the rescaled space, we have
\begin{equation}\label{csam}
\begin{split}
        \int_{\XX\setminus B_R(x)} p_1^i(x,y) |f_i(y)|^p\,\dd\mass_{t_i}^x(y)&\le \Vert f\Vert_{L^\infty(\mass)}^p\frac{1}{R^p}\int p_1^i(x,y)\frac{\dist(x,y)^p}{t_i^{p}}\,\dd\mass_{t_i}^x(y)\\&\le C \Vert f\Vert_{L^\infty(\mass)}^p\frac{1}{R^p} \int p_4^i(x,y)\,\dd\mass_{t_i}^x(y).
\end{split}
\end{equation}

We can use the Ascoli--Arzelà theorem and take a (non-relabeled) subsequence such that $f_i$
converge to $g$ locally uniformly, for some limit $g\in\LIP(\RR^n)$. As in \cite[Proposition 3.3]{Gor}, building upon \cite{Cheeger00},
\begin{equation}\label{eq:g_limit}
g(z)=\lip(f)(x)(z\cdot v),\quad\text{ for some }v\in\mathbb S^{n-1}.
\end{equation}
Therefore, $|f_i|^p$ converge locally uniformly to $|g|^p$. Using the theory of pmGH convergence (in particular,
exploiting also the doubling property of the measures involved), we see that $\nchi_{B^i_R(x)}|f_i|^p\mass_{t_i}^x$
converges to $\nchi_{\tilde B_R(0)}|g|^p\tilde\LL^n$ in duality with $C_{\rm bs}(\ZZ)$. Now we show that
\begin{equation}\label{eq:integr_limit_aux3aa}
    \heat_1^i\big(\nchi_{B_R^i(x)}|f_i|^p\big)(x)\rightarrow \tilde\heat_1\big(\nchi_{\tilde B_R(0)}|g|^p\big)(0).
\end{equation}
The claim follows as in the first part of the proof of \cite[Proposition 4.12]{ambrosio2018rigidity}, we sketch the argument.
First, by \cite[Lemma 4.11]{ambrosio2018rigidity} it holds that
\begin{equation}\notag
\liminf_{i\to\infty}\heat_1^i\big(\nchi_{B_R^i(x)}|f_i|^p\big)(x)\ge
\tilde\heat_1\big(\nchi_{\tilde B_R(0)}|g|^p\big)(0).
\end{equation}
Now, if $L$ denotes the Lipschitz constant of $f$ (recall that we are assuming $f(x)=0$), by the fact that
$\mass_{t_i}^x\rightharpoonup\tilde\LL^n$ in duality with \(C_{\rm bs}(\ZZ)\) we have that, in duality with \(C_{\rm bs}(\ZZ)\),
$$
\bigg(L R-\nchi_{B_R^i(x)}|f_i|^p\bigg)\mass_{t_i}^x\rightharpoonup \bigg(L R-\nchi_{\tilde B_R(0)}|g|^p\bigg)\tilde\LL^n
$$
and that the measures involved above are non-negative. As above, we infer that
\begin{equation}\notag
\liminf_{i\to\infty}\heat_1^i\bigg(L R-\nchi_{B_R^i(x)}|f_i|^p\bigg)(x)\ge
\tilde\heat_1\bigg(L R-\nchi_{\tilde B_R(0)}|g|^p\bigg)(0),
\end{equation}
thus accordingly
\begin{equation}\notag
\limsup_{i\to\infty}\heat_1^i\big(\nchi_{B_R^i(x)}|f_i|^p\big)(x)\le
\tilde\heat_1\big(\nchi_{\tilde B_R(0)}|g|^p\big)(0),
\end{equation}
so that \eqref{eq:integr_limit_aux3aa} is proved.

Using \eqref{csam} and \eqref{eq:integr_limit_aux3aa}, letting then $R\nearrow+\infty$, we have that
$$
\int p^{i}_1 (x,y)|f_i(y)|^p\,\dd\mass_{t_i}^x(y)\rightarrow\tilde\heat_1(|g|^p)(0).
$$
Hence we compute, using \eqref{eq:g_limit} for the first equality, a change-of-variable for the second one, and Fubini's theorem for the third one,
\begin{align*}
    \tilde\heat_1(|g|^p)(0)&=\lip(f)(x)^p\int_{\RR^n} \frac{1}{(4\pi)^{n/2}}e^{-\frac{|z|^2}{4}} |z\,\cdot\,v|^p\,\dd\LL^n(z)\\
    &=\lip(f)(x)^p\int_{\RR^n} \frac{1}{(4\pi)^{n/2}}e^{-\frac{|z|^2}{4}} |z_1|^p\,\dd\LL^n(z_1,\ldots,z_n)\\&=
    \lip(f)(x)^p\frac{1}{\sqrt{4\pi}}\int_\RR  |s|^pe^{-s^2/4}\,\dd\mathcal L^1(s)=
    \lip(f)(x)^p\frac{2^p}{\sqrt{\pi}}\Gamma\bigg(\frac{p+1}{2}\bigg),
\end{align*}
so that the proof is concluded.
\end{proof}

\bigskip

Now we pass to the case $p=1$, which is more difficult to handle than the previous one.
The following lemma is inspired by the proof of \cite[Theorem 2.13]{BMP}, obtained with a blow-up argument.
We will also rely on a blow-up argument, but facing the additional difficulty of the lack of a linear structure.
\begin{lem}\label{lem:integr_limit}
Let \((\XX,\dist,\mass)\) be an \(\RCD(K,N)\) space and let \(E\subseteq\XX\) be a set of finite perimeter. Then, for every $x\in\mathcal F E$
such that \eqref{zuppa1eq} and \eqref{zuppa2eq} hold at \(x\) (and thus for \(|\DIFF\nchi_E|\)-a.e.\ \(x\)),
\begin{equation}\label{eq:ptwse_limit}
\lim_{t\searrow 0}t\heat_{t^2}|\nabla\heat_{t^2}\nchi_E|(x)=\frac{1}{\sqrt{8\pi}}.
\end{equation}
In particular, it holds that
\begin{equation}\label{eq:integr_limit}
\lim_{t\searrow 0}t\int\heat_{t^2}|\nabla\heat_{t^2}\nchi_E|\,\dd|\DIFF\nchi_E|=\frac{1}{\sqrt{8\pi}}|\DIFF\nchi_E|(\XX).
\end{equation}
\end{lem}
\begin{proof}
Fix $x$ as in the statement and \((t_i)_i\subseteq(0,1)\) with \(t_i\searrow 0\). We use the same convergence of rescaled spaces as in Lemma \ref{Gorny2}, and we keep the same notation. Recall in addition that \(|\nabla^i\#|=t_i|\nabla\#|\), hence
\begin{equation}\label{eq:integr_limit_aux1}\begin{split}
t_i|\nabla\heat_{t_i^2}\nchi_E|(y)&=t_i\bigg|\nabla\int p_{t_i^2}(\cdot,z)\nchi_E(z)\,\dd\mass(z)\bigg|(y)
=\bigg|\nabla^i\int p_1^i(\cdot,z)\nchi_E(z)\,\dd\mass_{t_i}^x(z)\bigg|(y)\\
&=|\nabla^i\heat_1^i\nchi_E|(y),\quad\text{ for }\mass\text{-a.e.\ }y\in\XX.
\end{split}\end{equation}
Now observe that for any given radius \(R>0\) we can estimate
\begin{equation}\label{eq:integr_limit_aux2}\begin{split}
\int_{\XX\setminus B_R^i(x)}p_1^i(x,y)|\nabla^i\heat_1^i\nchi_E|(y)\,\dd\mass_{t_i}^x(y)\overset{\eqref{eq:Linfty_to_Lip}}\leq&
\frac{\|\heat_1^i\nchi_E\|_{L^\infty(\mass_{t_i}^x)}}{\sqrt 2 e^{K^-}}\int_{\XX\setminus B_R^i(x)}p_1^i(x,y)\,\dd\mass_{t_i}^x(y)\\
\overset{\eqref{eq:maximum_principle}}\leq&\frac{1}{\sqrt 2 e^{K^-}}\int_{\XX\setminus B_R^i(x)}p_1^i(x,y)\,\dd\mass_{t_i}^x(y)\\
\overset{\phantom{\eqref{eq:Linfty_to_Lip}}}=&\frac{1}{\sqrt 2 e^{K^-}}\int_{\XX\setminus B_{t_i R}(x)}p_{t_i^2}(x,y)\,\dd\mass(y)\\
\overset{\eqref{eq:heat_kernel_int_est}}\leq&\frac{C_3}{\sqrt 2 e^{K^-}}e^{-\frac{R^2}{24}}.
\end{split}\end{equation}

We also know from \cite[Proposition 4.12]{ambrosio2018rigidity} that
\(|\nabla^i\heat_1^i\nchi_E|\mass_{t_i}^x\rightharpoonup|\tilde\nabla\tilde\heat_1\nchi_H|\tilde{\mathcal L}^n\)
in duality with \(C_{\rm bs}(\ZZ)\). Given any \(R>0\), we also have that \(\nchi_{B_R^i(x)}|\nabla^i\heat_1^i\nchi_E|\mass_{t_i}^x
\rightharpoonup\nchi_{\tilde B_R(0)}|\tilde\nabla\tilde\heat_1\nchi_H|\tilde{\mathcal L}^n\) in duality with \(C_{\rm bs}(\ZZ)\). Also, by \eqref{eq:Linfty_to_Lip}, it holds that $|\nabla^i\heat_1^i\nchi_E|$,$|\tilde\nabla\tilde\heat_1\nchi_H|\le \frac{1}{\sqrt{2}e^{K^-}}$ almost everywhere, for every $i$.
Then, following the proof of \eqref{eq:integr_limit_aux3aa}, we can easily show that
\begin{equation}\label{eq:integr_limit_aux3}
\heat_1^i\big(\nchi_{B_R^i(x)}|\nabla^i\heat_1^i\nchi_E|\big)(x)\to
\tilde\heat_1\big(\nchi_{\tilde B_R(0)}|\tilde\nabla\tilde\heat_1\nchi_H|\big)(0),\quad\text{ as }i\to\infty.
\end{equation}
Notice that, as \(R\nearrow+\infty\), 
\begin{equation}\label{eq:integr_limit_aux4}
\tilde\heat_1\big(\nchi_{\tilde B_R(0)}|\tilde\nabla\tilde\heat_1\nchi_H|\big)(0)\to\tilde\heat_1\big(|\tilde\nabla\tilde\heat_1\nchi_H|\big)(0)
=\heat_1^{\RR^n}\big(|\nabla^{\RR^n}\heat_1^{\RR^n}\nchi_H|\big)(0)=\frac{1}{\sqrt{8\pi}},
\end{equation}
where the limit is due to monotone convergence and the last equality follows from the direct computation 
\begin{align*}
|\nabla^{\RR^n} \heat_1^{\RR^n} \nchi_H|&=\bigg|\nabla^{\RR^n} \int_{\{z_n>0\}}p_1^{\RR^n}(x,z)\,\dd\LL^n(z)\bigg|
=\bigg|\frac{\dd}{\dd x} \int_{\{z>0\}}p_1^{\RR}(x,z)\,\dd\LL^1(z)\bigg|\\
&=\bigg| \int_{\{z>0\}} \frac{\dd}{\dd x}p_1^{\RR}(x,z)\,\dd\LL^1(z)\bigg|
=\bigg|\int_{\{z>0\}} \frac{\dd}{\dd z}p_1^{\RR}(x,z)\,\dd\LL^1(z)\bigg|=p_1^{\RR}(x,0),
\end{align*}
which yields
$$
\heat_1^{\RR^n}\big(|\nabla^{\RR^n} \heat_1^{\RR^n}\nchi_H|\big)(0)=(\heat_1^{\RR} p_1^{\RR}(\,\cdot\,,0))(0)=p_2^\RR(0,0)=\frac{1}{\sqrt{8 \pi}}.
$$
Next, we can estimate
\[\begin{split}
&\bigg|\int p_1^i(x,y)|\nabla^i\heat_1^i\nchi_E|(y)\,\dd\mass_{t_i}^x(y)-\frac{1}{\sqrt{8\pi}}\bigg|\\
\overset{\eqref{eq:integr_limit_aux2}}\leq\,&\bigg|\int_{B_R^i(x)}p_1^i(x,y)|\nabla^i\heat_1^i\nchi_E|(y)\,\dd\mass_{t_i}^x(y)-\frac{1}{\sqrt{8\pi}}\bigg|
+\frac{C_3}{\sqrt 2 e^{K^-}}e^{-\frac{R^2}{24}}\\
\overset{\phantom{\eqref{eq:integr_limit_aux2}}}\leq\,&\bigg|\heat_1^i\big(\nchi_{B_R^i(x)}|\nabla^i\heat_1^i\nchi_E|\big)(x)-\frac{1}{\sqrt{8\pi}}\bigg|
+\frac{C_3}{\sqrt 2 e^{K^-}}e^{-\frac{R^2}{24}},
\end{split}\]
whence (by letting \(i\to\infty\)) it follows that
\[\begin{split}
&\limsup_{i\to\infty}\bigg|\int p_1^i(x,y)|\nabla^i\heat_1^i\nchi_E|(y)\,\dd\mass_{t_i}^x(y)-\frac{1}{\sqrt{8\pi}}\bigg|\\
\overset{\eqref{eq:integr_limit_aux3}}\leq\,&\bigg|\tilde\heat_1\big(\nchi_{\tilde B_R(0)}|\tilde\nabla\tilde\heat_1\nchi_H|\big)(0)
-\frac{1}{\sqrt{8\pi}}\bigg|+\frac{C_3}{\sqrt 2 e^{K^-}}e^{-\frac{R^2}{24}}.
\end{split}\]
Letting \(R\nearrow+\infty\) and exploiting \eqref{eq:integr_limit_aux4}, we deduce that
\(\int p_1^i(x,y)|\nabla^i\heat_1^i\nchi_E|(y)\,\dd\mass_{t_i}^x(y)\to\frac{1}{\sqrt{8\pi}}\) as \(i\to\infty\). Consequently, we can conclude that
\[
t_i\heat_{t_i^2}|\nabla\heat_{t_i^2}\nchi_E|(x)=t_i\int p_{t_i^2}(x,y)|\nabla\heat_{t_i^2}\nchi_E|(y)\,\dd\mass(y)
\overset{\eqref{eq:integr_limit_aux1}}=\int p_1^i(x,y)|\nabla^i\heat_1^i\nchi_E|(y)\,\dd\mass_{t_i}^x(y)
\]
converges to \(\frac{1}{\sqrt{8\pi}}\) as \(i\to\infty\). Since the sequence \(t_i\searrow 0\) was arbitrary, \eqref{eq:ptwse_limit} is finally proved.

Let us pass to the verification of \eqref{eq:integr_limit}.
Given any \(t>0\), we can estimate
\[
t\heat_{t^2}|\nabla\heat_{t^2}\nchi_E|(x)\overset{\eqref{eq:maximum_principle_RCD}}\leq t\||\nabla\heat_{t^2}\nchi_E|\|_{L^\infty(\mass)}
\overset{\eqref{eq:Linfty_to_Lip}}\leq\frac{1}{\sqrt 2 e^{K^-}},\quad\text{ for every }x\in\XX.
\]
Hence, using the dominated convergence theorem and taking \eqref{eq:ptwse_limit} and Proposition \ref{zuppa} into account, we get \eqref{eq:integr_limit}.
\end{proof}

The following lemma is one of the main results of the paper. In the case in which $E=F$, the conclusion comes from L'H\^{o}pital's rule and the blow-up computation of Lemma \ref{lem:integr_limit}. The computations with L'H\^{o}pital's rule have already been used in \cite{BMP}, but in our case we cannot exploit the linear structure of Carnot groups, so that we have to rely on the regularizing properties of the heat flow on $\RCD$ spaces to face additional difficulties. The case of possibly different $E$ and $F$ is obtained with a polarization argument and a careful computation: due to the lack of a linear structure, the strategy used to prove \cite[Theorem 3.1]{MPPP} seems not suitable to our case.
\begin{lem}\label{maincomp}
		Let $(\XX,\dist,\mass)$ be an $\RCD(K,N)$ space of essential dimension \(n\). Let $E,F\subseteq\XX$ be two sets of finite perimeter and finite measure. 
		Then
		\begin{equation}\label{vfsvdnjc}
			\lim_{t\searrow 0}\frac{1}{\sqrt{t}}\int (\nchi_E-\heat_t\nchi_E)\nchi_F\,\dd\mass= \frac{1}{\sqrt{\pi}}\frac{\omega_{n-1}}{\omega_n}\int_{\FF E\cap\FF F}\nu_E\,\cdot\,\nu_F\,\dd\HH^h.
		\end{equation}
In particular, it holds that
\begin{equation}\label{eq:main_comp_E=F}
\lim_{t\searrow 0}\frac{1}{\sqrt{t}}\int (\nchi_E-\heat_t\nchi_E)\nchi_E\,\dd\mass=
\frac{1}{\sqrt\pi}|\DIFF\nchi_E|(\XX)=\frac{1}{\sqrt\pi}\frac{\omega_{n-1}}{\omega_n}\HH^h(\mathcal F E).
\end{equation}
	\end{lem}
	\begin{proof}
	We have to compute
	$$
	\lim_{t\searrow 0}\frac{\int (\nchi_E-\heat_{t^2}\nchi_E)\nchi_F\,\dd\mass}{t}
	$$
	for $E$ and $F$ sets of finite perimeter and finite measure.
	Notice that as $E$ and $F$ have finite measure, $\int (\nchi_E-\heat_{t^2}\nchi_E)\nchi_F\,\dd\mass\rightarrow 0$ as $t\searrow 0$,
	so that by L'H\^{o}pital's rule we reduce ourselves to compute
	
$$\lim_{t\searrow 0}\bigg(-2t{\int\nchi_F \Delta\heat_{t^2}\nchi_E\,\dd\mass}\bigg)=\lim_{t\searrow 0} 2t{\int\nabla\heat_{t^2/2}\nchi_F\,\cdot\,\nabla \heat_{t^2/2} \nchi_E\,\dd\mass}.$$
Therefore, we conclude that the left-hand side of \eqref{vfsvdnjc} is equal to
$$
\lim_{t\searrow 0}\sqrt{8}\,t{\int\nabla\heat_{t^2}\nchi_F\,\cdot\,\nabla \heat_{t^2} \nchi_E\,\dd\mass}.$$
		For $E,F\subseteq\XX$ sets of finite perimeter and finite measure, we write $$g_t(E,F)\defeq \sqrt{8}\, t\int\nabla\heat_{t^2}\nchi_E\,\cdot\,\nabla\heat_{t^2}\nchi_F\,\dd\mass,\quad\text{ for }t>0.$$
		If $E=F$, we simply write $g_t(E)$ instead of $g_t(E,E)$.
		Our aim is then to study $\lim_{t\searrow 0} g_t(E,F)$, where $E$ and $F$ are sets of finite perimeter and finite measure. We start by studying two particularly simple cases and then we treat the general case.
		\\\textsc{Case 1}: $E=F$. Applying \cite[Lemma 2.7]{bru2019rectifiability}, we get
		\begin{equation}\label{eq:comp_TV_claim1_aux2}\begin{split}
				&\bigg|t\int|\nabla\heat_{t^2}\nchi_E|^2\,\dd\mass-te^{-Kt^2}\int|\nabla\heat_{t^2}\nchi_E|\,\heat_{t^2}^*|\DIFF\nchi_E|\,\dd\mass\bigg|\\
				\overset{\phantom{\eqref{eq:Linfty_to_Lip}}}\leq\,&\frac{t}{e^{Kt^2}}\int|\nabla\heat_{t^2}\nchi_E|
				\bigg|e^{Kt^2}\frac{|\nabla\heat_{t^2}\nchi_E|}{\heat_{t^2}^*|\DIFF\nchi_E|}-1\bigg|\,\heat_{t^2}^*|\DIFF\nchi_E|\,\dd\mass\\
				\overset{\eqref{eq:Linfty_to_Lip}}\leq\,&\frac{1}{\sqrt 2\,e^{2K^-}}
				\int\bigg|e^{Kt^2}\frac{|\nabla\heat_{t^2}\nchi_E|}{\heat_{t^2}^*|\DIFF\nchi_E|}-1\bigg|\,\heat_{t^2}^*|\DIFF\nchi_E|\,\dd\mass\to 0,
				\quad\text{ as }t\searrow 0.
		\end{split}\end{equation}
		Recall that \(|\nabla\heat_{t^2}\nchi_E|\in L^\infty(\mass)\) holds for every \(t>0\) by \eqref{eq:Linfty_to_Lip}. Hence, Lemma \ref{lem:integr_limit} yields
		\begin{equation}\label{eq:comp_TV_claim1_aux3}
			\lim_{t\searrow 0}t\int|\nabla\heat_{t^2}\nchi_E|\,\heat_{t^2}^*|\DIFF\nchi_E|\,\dd\mass=
			\lim_{t\searrow 0}t\int\heat_{t^2}|\nabla\heat_{t^2}\nchi_E|\,\dd|\DIFF\nchi_E|=\frac{1}{\sqrt{8\pi}}|\DIFF\nchi_E|(\XX).
		\end{equation}
		Then, by combining  \eqref{eq:comp_TV_claim1_aux2},  \eqref{eq:comp_TV_claim1_aux3} and finally \eqref{reprformula}, 
		we obtain that 
		$$
		\lim_{t \searrow 0} g_t(E)=\sqrt{8}\frac{1}{\sqrt{8\pi}}|\DIFF\nchi_E|(\XX)=\frac{1}{\sqrt{\pi}}\frac{\omega_{n-1}}{\omega_n} \HH^h(\FF E).
		$$
		\\\textsc{Case 2}: $E\cap F=\varnothing$.  We start noticing that $\nchi_{E\cup F}=\nchi_E+\nchi_F$, so that the linearity of the heat flow and of the gradient yields
		\begin{align*}
			g_t(E\cup F)=g_t(E)+g_t(F)+2g_t(E,F).
		\end{align*}
		Therefore, by the first case,
		$$
		\lim_{t\searrow 0} g_t(E,F)=\frac{1}{\sqrt{\pi}}\frac{|\DIFF\nchi_{E\cup F}|(\XX)-|\DIFF\nchi_{E}|(\XX)-|\DIFF\nchi_{F}|(\XX)}{2}.
		$$
		As $E$ and $F$ are disjoint, by \cite[Theorem 4.11]{bru2021constancy},  \cite[Lemma 3.25]{BG22}, item 2) of Proposition \ref{zuppa}, and \eqref{reprformula} we have that
		\begin{align*}
			|\DIFF\nchi_{E\cup F}|(\XX)&=\frac{\omega_{n-1}}{\omega_n} \big(\HH^h(\FF E\setminus \FF F)+\HH^h(\FF F\setminus \FF E)\big)\\&=\frac{\omega_{n-1}}{\omega_n} \big(\HH^h(\FF E)+\HH^h(\FF F)-2\HH^h(\FF E\cap \FF F)\big),
		\end{align*}
		so that, by \eqref{reprformula} again,
		$$
		\lim_{t\searrow 0} g_t(E,F)=-\frac{1}{\sqrt{\pi}}\frac{\omega_{n-1}}{\omega_n} \HH^h(\FF E\cap \FF F).
		$$
		\\\textsc{Case 3: the general case.} Write $E=(E\cap F)\cup (E\setminus F)$ and $F=(F\cap E)\cup (F\setminus E)$; notice that unions are disjoint. Exploiting linearity as before, we write
		\begin{align*}
			g_t(E,F)=g_t(E\cap F)+g_t(E\cap F,F\setminus E)+g_t(E\setminus F,F\cap E)+g_t(E\setminus F,F\setminus E).
		\end{align*}
		Using the two cases treated above, we have that
		\begin{equation}\label{eq:gen_formula_g_t}\begin{split}
			\lim_{t\searrow 0} g_t(E,F)=&\frac{1}{\sqrt{\pi}}\frac{\omega_{n-1}}{\omega_n} \big(\HH^h(\FF(E\cap F))-\HH^h(\FF(E\cap F)\cap \FF(F\setminus E) ) \\&-\HH^h(\FF(E\setminus F)\cap \FF(F\cap E) ) -\HH^h(\FF(E\setminus F)\cap \FF(F\setminus E) ) \big).
		\end{split}\end{equation}
		By density arguments we obtain that, up to $\HH^h$-negligible sets, (see also \cite[Lemma 3.25]{BG22})
		\begin{align*}
			\FF(E\cap F)&=\big(\FF E\cap F^{(1)}\big)\cup \big(\FF F\cap E^{(1)}\big)\cup  \big(\FF E\cap \FF F\cap (E\Delta F)^{(0)}\big),\\
			\FF(F\setminus E)&=\big(\FF F\cap E^{(0)}\big)\cup \big(\FF E\cap F^{(1)}\big)\cup  \big(\FF E\cap \FF F\cap (E\Delta F)^{(1)}\big),\\
			\FF(E\setminus F)&=\big(\FF E\cap F^{(0)}\big)\cup \big(\FF F\cap E^{(1)}\big)\cup  \big(\FF E\cap \FF F\cap (E\Delta F)^{(1)}\big),
		\end{align*}
		where all unions are disjoint.
		We use these identities to compute, up to $\HH^h$-negligible sets,
		\begin{align*}
			\FF(E\cap F)\cap \FF(F\setminus E)&=\FF E\cap F^{(1)},\\
			\FF(E\setminus F)\cap \FF(F\cap E)&=\FF F\cap E^{(1)},\\
			\FF(E\setminus F)\cap \FF(F\setminus E)&=\FF E\cap \FF F\cap (E\Delta F)^{(1)}.
		\end{align*}
		We compute then
		\begin{align*}
			&\HH^h(\FF(E\cap F))-\HH^h(\FF(E\cap F)\cap \FF(F\setminus E) )- \HH^h(\FF(E\setminus F)\cap \FF(F\cap E) ) \\&\qquad\qquad-\HH^h(\FF(E\setminus F)\cap \FF(F\setminus E) ) 
			\\&\qquad=\HH^h\big(\FF E\cap F^{(1)}\big)+\HH^h \big(\FF F\cap E^{(1)}\big)+\HH^h  \big(\FF E\cap \FF F\cap (E\Delta F)^{(0)}\big)
			\\&\qquad\qquad-\HH^h\big(\FF E\cap F^{(1)}\big)-\HH^h\big(\FF F\cap E^{(1)}\big)-\HH^h\big(\FF E\cap \FF F\cap (E\Delta F)^{(1)}\big)
			\\&\qquad= \HH^h  \big(\FF E\cap \FF F\cap (E\Delta F)^{(0)}\big)-\HH^h\big(\FF E\cap \FF F\cap (E\Delta F)^{(1)}\big)
			\\&\qquad=\int_{\FF E\cap \FF F}\nu_E\,\cdot\,\nu_F\,\dd\HH^h,
		\end{align*}
		where in the last equality we used \cite[Lemma 3.25]{BG22}. Therefore, by recalling \eqref{eq:gen_formula_g_t} we conclude that the statement
		holds for any \(E,F\subseteq\XX\) of finite perimeter and finite measure.
\end{proof}

Having Lemma \ref{maincomp} at our disposal, there is no effort in obtaining Theorem \ref{cdsnjca} below.
\begin{thm}\label{cdsnjca}
Let \((\XX,\dist,\mass)\) be an \(\RCD(K,N)\) space. Let \(E\subseteq\XX\) be a set of finite perimeter such that either $\mass(E)<\infty$ or $\mass(\XX\setminus E)<\infty$.
Then 
\begin{equation}\label{eq:main_for_sets}
\lim_{t\searrow 0}\frac{1}{\sqrt t}\int\!\!\!\int p_t(x,y)|\nchi_E(x)-\nchi_E(y)|\,\dd\mass(x)\,\dd\mass(y)=\frac{2}{\sqrt{\pi}}|\DIFF\nchi_E|(\XX).
\end{equation}
\end{thm}
\begin{proof}
We can assume with no loss of generality that $\mass(E)<\infty$. We compute (all the integrands are positive)
\begin{equation}\label{vjns}
	\begin{split}
		&\int\!\!\!\int p_t(x,y)|\nchi_E(y)-\nchi_E(x)|\,\dd\mass(y)\,\dd\mass(x)\\
		=\,&\int(1-\nchi_E(x))\int_E p_t(x,y)\,\dd\mass(y)\,\dd\mass(x)+\int\nchi_E(x)\int_{\XX\setminus E}p_t(x,y)\,\dd\mass(y)\,\dd\mass(x)\\
		=\,&\int\nchi_{\XX\setminus E}\,\heat_t\nchi_E\,\dd\mass+\int\nchi_E\,\heat_t\nchi_{\XX\setminus E}\,\dd\mass
		=2	\int_{\XX\setminus E}\heat_t\nchi_E\,\dd\mass=2\int (1-\nchi_E)\heat_t\nchi_E\,\dd\mass\\=\,&2\bigg(\mass(E)-\int \nchi_E\,\heat_t\nchi_E\,\dd\mass\bigg)=2\int(\nchi_E-\heat_t\nchi_E)\nchi_E\,\dd\mass.
\end{split}
\end{equation}
We obtain \eqref{eq:main_for_sets} by dividing the above equation by $\sqrt t $, letting $t\searrow 0$, and using \eqref{eq:main_comp_E=F}.
\end{proof}

As a technical tool, we need the following easy computation, obtained via classical techniques. Note the role played by the regularizing properties of the heat flow on $\RCD$ spaces.
\begin{lem}\label{Lebesgue}
Let $(\XX,\dist,\mass)$ be an $\RCD(K,N)$ space and $E,F\subseteq\XX$ two sets of finite perimeter and finite measure. Then
\begin{equation}\notag
\bigg|\int(\heat_t\nchi_E-\nchi_E)\nchi_F\,\dd\mass\bigg|\le 2 e^{K^-(\frac{t}{2}-1)}|\DIFF\nchi_E|(\XX)\sqrt t.
\end{equation}
\end{lem}
\begin{proof}We compute
\[\begin{split}
\bigg|\int(\heat_t\nchi_E-\nchi_E)\nchi_F\,\dd\mass\bigg|&\overset{\phantom{\eqref{eq:Linfty_to_Lip}}}
=\bigg|\int_0^t\frac{\dd}{\dd s}\int(\heat_s\nchi_E-\nchi_E)\nchi_F\,\dd\mass\,\dd s\bigg|=
\bigg|\int_0^t\!\!\int\nchi_F\Delta\heat_s\nchi_E\,\dd\mass\,\dd s\bigg|\\
&\overset{\phantom{\eqref{eq:Linfty_to_Lip}}}=\bigg|\int_0^t\!\!\int\nabla\heat_{s/2}\nchi_F\cdot\nabla\heat_{s/2}\nchi_{ E}\,\dd\mass\,\dd s\bigg|\\
&\overset{\phantom{\eqref{eq:Linfty_to_Lip}}}\leq\int_0^t\!\!\int|\nabla\heat_{s/2}\nchi_F||\nabla\heat_{s/2}\nchi_E|\,\dd\mass\,\dd s\\
&\overset{\eqref{eq:Linfty_to_Lip}}\leq\frac{1}{e^{K^-}}\int_0^t\frac{1}{\sqrt s}\int|\nabla\heat_{s/2}\nchi_E|\,\dd\mass\,\dd s\\
&\overset{\eqref{eq:Bakry-Emery}}\leq\frac{1}{e^{K^-}}\int_0^t\frac{e^{-Ks/2}}{\sqrt s}\int\heat_{s/2}^*|\DIFF\nchi_E|\,\dd\mass\,\dd s\\
&\overset{\phantom{\eqref{eq:Linfty_to_Lip}}}\leq e^{K^-(\frac{t}{2}-1)}\int_0^t\frac{e^{-Ks/2}}{\sqrt s}\int\heat_{s/2}^*|\DIFF\nchi_E|\,\dd\mass\,\dd s\\
&\overset{\eqref{eq:dual_mass_preserv}}=e^{K^-(\frac{t}{2}-1)}|\DIFF\nchi_E|(\XX)\int_0^t\frac{1}{\sqrt s}\,\dd s\\
&\overset{\phantom{\eqref{eq:Linfty_to_Lip}}}=2 e^{K^-(\frac{t}{2}-1)}|\DIFF\nchi_E|(\XX)\sqrt t,
\end{split}\]
which proves the claim.
\end{proof}

The following remark is borrowed from \cite{BMP}.
\begin{rem}\label{rem:Fubini_argument}
Given a metric measure space \((\XX,\dist,\mass)\) and a function \(f\in L^1(\mass)\), it holds that
\begin{equation}\label{eq:Fubini_argument}
\int_\RR|\nchi_{\{f>t\}}(x)-\nchi_{\{f>t\}}(y)|\,\dd t=|f(x)-f(y)|,\quad\text{ for }(\mass\otimes\mass)\text{-a.e.\ }(x,y)\in\XX\times\XX.
\end{equation}
Indeed, fixed a Borel representative \(\bar f\colon\XX\to\RR\) of \(f\) and \(x,y\in\XX\) with \(\bar f(y)\leq\bar f(x)\),
we have that \(t\in\RR\) satisfies \(|\nchi_{\{\bar f>t\}}(x)-\nchi_{\{\bar f>t\}}(y)|=1\) if \(\bar f(y)\leq t<\bar f(x)\),
and \(|\nchi_{\{\bar f>t\}}(x)-\nchi_{\{\bar f>t\}}(y)|=0\) otherwise. Therefore, we can compute
\[
\int_\RR|\nchi_{\{\bar f>t\}}(x)-\nchi_{\{\bar f>t\}}(y)|\,\dd t=\mathcal L^1\big(\big[\bar f(y),\bar f(x)\big)\big)=|\bar f(x)-\bar f(y)|,
\]
whence it follows that \(\int_\RR|\nchi_{\{\bar f>t\}}(x)-\nchi_{\{\bar f>t\}}(y)|\,\dd t=|\bar f(x)-\bar f(y)|\) for every \(x,y\in\XX\).
Finally, an application of Fubini's theorem ensures that \eqref{eq:Fubini_argument} is verified, as we claimed.\fr
\end{rem}
\subsection{Proof of the main results}
We are now finally in a position to prove the main results of this paper.
\begin{proof}[Proof of Theorem \ref{main1}] Lemma \ref{membership} i) says that if $f\notin W^{1,p}(\XX)$, then the statement holds.
Therefore, in what follows we assume $f\in W^{1,p}(\XX)$. As in the proof of \cite[Theorem 3.5]{Gor} (recall Lemma \ref{Gorny2} and
\eqref{eq:|Df|=lipf}), we see that it is enough to show that for any $f\in \LIP_{\rm bs}(\XX)$ it holds that
\begin{equation*}
\lim_{t\searrow 0}\frac{1}{t^{p/2}}\int\!\!\!\int p_t(x,y)|f(x)-f(y)|^p\,\dd\mass(x)\,\dd\mass(y)
=\frac{2^p}{\sqrt{\pi}}\Gamma\bigg(\frac{p+1}{2}\bigg) \int \lip(f)^p\,\dd\mass.
\end{equation*}
Then let $f\in \LIP_{\rm bs}(\XX)$ be fixed. We take a ball $B$ that contains the open $1$-neighborhood of the support of $f$.
Let $L$ denote the global Lipschitz constant of $f$. We will perform a change-of-variable, replacing $t$ with $t^2$. Then
\begin{align*}
&\frac{1}{t^{p}}\int\!\!\!\int p_{t^2}(x,y)|f(x)-f(y)|^p\,\dd\mass(x)\,\dd\mass(y)=\\
&\frac{1}{t^{p}}\int_{B}\!\int p_{t^2}(x,y)|f(x)-f(y)|^p\,\dd\mass(x)\,\dd\mass(y)
+\frac{1}{t^{p}}\int_{\XX\setminus B}\!\int p_{t^2}(x,y)|f(x)-f(y)|^p\,\dd\mass(x)\,\dd\mass(y).
\end{align*}
We treat the two summands on the right-hand side separately. For what concerns the first summand, notice that using \eqref{eqcomoda}
we obtain for \(\mass\)-a.e.\ \(y\in\XX\) that
\begin{align*}
\frac{1}{t^p}\int p_{t^2}(x,y)|f(x)-f(y)|^p\,\dd\mass(x)&\le L^p \int p_{t^2}(x,y)\frac{\dist(x,y)^p}{t^p}\,\dd\mass(x)\\
&\le C L^p \int p_{4 t^2}(x,y)\,\dd\mass(x)=C L^p.
\end{align*}
Therefore, by dominated convergence, and taking into account Lemma \ref{Gorny2}, we have that
\[
 \lim_{t\searrow 0} \frac{1}{t^{p}}\int_{B}\int p_{t^2}(x,y)|f(x)-f(y)|^p\,\dd\mass(x)\,\dd\mass(y)
 =\frac{2^p}{\sqrt{\pi}}\Gamma\bigg(\frac{p+1}{2}\bigg)\int \lip(f)^p\,\dd\mass.
\]
Now we treat the second summand. Notice that if $y\in \XX\setminus B$ and $x$ does not belong to the support of $f$, then $f(x)=f(y)=0$,
thus the following integrals are only over the set of \((x,y)\) such that $\dist(x,y)\ge 1$. Hence, using also \eqref{eqcomoda}
and the contractivity of the heat flow, we get
\begin{align*}
&\frac{1}{t^{p}}\int_{\XX\setminus B}\int p_{t^2}(x,y)|f(x)-f(y)|^p\,\dd\mass(x)\,\dd\mass(y)\\
\le\,& Ct\int\!\!\!\int p_{t^2}(x,y)\frac{\dist(x,y)^p}{t^p}\frac{\dist(x,y)}{t}(|f(x)|^p+|f(y)|^p)\,\dd\mass(x)\,\dd\mass(y)\\
\le\,& Ct\int\!\!\!\int p_{4 t^2}(x,y)(|f(x)|^p+|f(y)|^p)\,\dd\mass(x)\,\dd\mass(y)\\
\le\,& Ct\Vert f\Vert_{L^p(\mass)}^p,
\end{align*}
which converges to $0$ as $t\searrow 0$. Consequently, the proof is concluded.
\end{proof}
\begin{proof}[Proof of Theorem \ref{main2}]
Lemma \ref{membership} ii) says that if $f\notin \BV(\XX)$, then the statement holds. Therefore, in what follows we assume $f\in \BV(\XX)$.
We argue via coarea and integration via Cavalieri's formula, as done in the references \cite[Theorem 2.14]{BMP} and \cite[Theorem 4.3]{MPPP},
building upon Theorem \ref{cdsnjca}. Let $f\in \BV(\XX)$ be given. Recall that \eqref{vjns} and Lemma \ref{Lebesgue} imply that 
\begin{equation}\label{casjcan}
		\frac{1}{\sqrt{t}}\int\!\!\!\int p_{t}(x,y)|\nchi_{E}(x)-\nchi_{E}(y)|\,\dd\mass(x)\,\dd\mass(y)\le 4 e^{K^-(\frac{t}{2}-1)}|\DIFF\nchi_E|(\XX)
\end{equation}
for every $E\subseteq\XX$ of finite perimeter and finite measure, for every $t>0$. Notice that this holds even if $E$ has infinite measure, but $\XX\setminus E$ has finite measure.

Now fix a function \(f\in\BV(\XX)\). For any real number \(s\in\RR\), pick a Borel representative \(E_s\subseteq\XX\) of \(\{f>s\}\).
Fix any \((t_i)_i\subseteq(0,1)\) with \(t_i\searrow 0\). 
 Hence, by using the dominated convergence theorem, we obtain that
\begin{equation}\label{eq:conseq_DCT}\begin{split}
&\int_\RR\lim_{i\to\infty}\frac{1}{\sqrt{t_i}}\int\!\!\!\int p_{t_i}(x,y)|\nchi_{E_s}(x)-\nchi_{E_s}(y)|\,\dd\mass(x)\,\dd\mass(y)\,\dd s\\
=\,&\lim_{i\to\infty}\frac{1}{\sqrt{t_i}}\int_\RR\int\!\!\!\int p_{t_i}(x,y)|\nchi_{E_s}(x)-\nchi_{E_s}(y)|\,\dd\mass(x)\,\dd\mass(y)\,\dd s.
\end{split}\end{equation}
We can use dominated convergence thanks to \eqref{casjcan} together with the coarea formula; the latter ensures that the function
\(s\mapsto|\DIFF\nchi_{E_s}|(\XX)\) is integrable. Finally, we can compute
\[\begin{split}
|\DIFF f|(\XX)&=\int_\RR|\DIFF\nchi_{E_s}|(\XX)\,\dd s\qquad\text{(by coarea formula)}\\
&=\int_\RR\lim_{i\to\infty}\frac{1}{2}\sqrt{\frac{\pi}{t_i}}\int\!\!\!\int p_{t_i}(x,y)|\nchi_{E_s}(x)-\nchi_{E_s}(y)|\,\dd\mass(x)\,\dd\mass(y)\,\dd s
\qquad\text{(by \eqref{eq:main_for_sets})}\\
&=\lim_{i\to\infty}\frac{1}{2}\sqrt{\frac{\pi}{t_i}}\int_\RR\int\!\!\!\int p_{t_i}(x,y)|\nchi_{E_s}(x)-\nchi_{E_s}(y)|\,\dd\mass(x)\,\dd\mass(y)\,\dd s
\qquad\text{(by \eqref{eq:conseq_DCT})}\\
&=\lim_{i\to\infty}\frac{1}{2}\sqrt{\frac{\pi}{t_i}}\int\!\!\!\int p_{t_i}(x,y)\int_\RR|\nchi_{E_s}(x)-\nchi_{E_s}(y)|\,\dd s\,\dd\mass(x)\,\dd\mass(y)
\qquad\text{(by Fubini)}\\
&=\lim_{i\to\infty}\frac{1}{2}\sqrt{\frac{\pi}{t_i}}\int\!\!\!\int p_{t_i}(x,y)|f(x)-f(y)|\,\dd\mass(x)\,\dd\mass(y).
\qquad\text{(by Remark \ref{rem:Fubini_argument})}
\end{split}\]
By the arbitrariness of \(t_i\searrow 0\), this proves the statement.
\end{proof}

Now we prove Theorem \ref{thm:main3} via coarea and integration via Cavalieri's formula, as done in the reference \cite[Theorem 4.3]{MPPP}, building upon Lemma \ref{maincomp}.
\begin{proof}[Proof of Theorem \ref{thm:main3}]
	First, we write $f=f^+ -f^-$, where $f^+\defeq f\vee 0$ and $f^-\defeq(-f)\vee 0$. Thanks to the coarea formula, we can apply \cite[Lemma 3.24]{BG22} and infer that $\nu_{f^\pm}=\nu_{\pm f}$ $|\DIFF f^\pm|$-a.e. Also, a direct computation yields that $f^\wedge=(f^+)^\wedge-(f^-)^{\vee}$  and $f^\vee=(f^+)^\vee-(f^-)^{\wedge}$. Therefore, by linearity, we can assume that $f\ge 0$ $\mass$-a.e. We repeat the same argument for $g$ to see that we can assume that also $g\ge 
	0$ $\mass$-a.e. Up to scaling and adding a constant, we assume that $0\le g\le 1$ $\mass$-a.e. We can compute
$$
\int (f-\heat_t f) g\,\dd\mass=\int\!\!\!\int p_t(x,y) g(x)(f(x)-f(y))\,\dd\mass(x)\,\dd\mass(y).
$$	
Let now $E_s\defeq\{f>s\}$ and $F_s\defeq \{g>s\}$ and notice that for \(\mathcal L^1\)-a.e.\ $\sigma,\tau\in\RR$ it holds that
\begin{equation}\label{cnados}
\begin{split}
&\frac{1}{\sqrt t}\int\!\!\!\int \big|p_t(x,y)\nchi_{F_\tau}(x) (\nchi_{E_\sigma}(x)-\nchi_{E_\sigma}(y))\big|\,\dd\mass(x)\,\dd\mass(y)\\
\le\,&\frac{1}{\sqrt t}\int\!\!\!\int p_t(x,y)\big|\nchi_{E_\sigma}(x)-\nchi_{E_\sigma}(y)\big|\,\dd\mass(x)\,\dd\mass(y)
\overset{\eqref{casjcan}}\le 4e^{K^-(\frac{t}{2}-1)}|\DIFF\nchi_{E_\sigma}|(\XX).
\end{split}
\end{equation}
By Cavalieri's formula, we can write 
$$
p_t(x,y) g(x)(f(x)-f(y))=\int_0^\infty\!\!\!\int_0^1 p_t(x,y)\nchi_{F_\tau}(x) (\nchi_{E_\sigma}(x)-\nchi_{E_\sigma}(y))\,\dd\tau\,\dd\sigma,
$$
so that, by Fubini's theorem (whose application is justified by \eqref{cnados} and coarea), we obtain
\begin{align*}
\frac{1}{\sqrt t}\int (f-\heat_t f)g\,\dd\mass&=\frac{1}{\sqrt t}\int_0^\infty\!\!\!\int_0^1\!\!\!\int\!\!\!\int
p_t(x,y)\nchi_{F_\tau}(x)(\nchi_{E_\sigma}(x)-\nchi_{E_\sigma}(y))\,\dd\mass(x)\,\dd\mass(y)\,\dd\tau\,\dd\sigma\\&
=\frac{1}{\sqrt t}\int_0^\infty\!\!\!\int_0^1\!\!\!\int\nchi_{F_\tau}(\nchi_{E_\sigma}-\heat_t\nchi_{E_\sigma})\,\dd\mass\,\dd\tau\,\dd\sigma.
\end{align*}

Lemma \ref{Lebesgue} together with coarea again justify an application of the dominated convergence theorem in the limit as $t\searrow 0$ in the equation above, so that, by Lemma \ref{maincomp},
\begin{equation}\label{eq:proof_main3_1}\begin{split}
	\lim_{t\searrow 0} \frac{1}{\sqrt t}\int (f-\heat_t f) g\,\dd\mass&=
	\frac{1}{\sqrt{\pi}}\frac{\omega_{n-1}}{\omega_n}\int_0^\infty\!\!\!\int_0^1\!\!\!\int_{\FF {E_\sigma}\cap\FF {F_\tau}}\nu_{E_\sigma}\cdot\nu_{F_\tau}\dd\HH^h\,\dd\tau\,\dd\sigma\\
	&=\frac{1}{\sqrt{\pi}}\frac{\omega_{n-1}}{\omega_n}\int_0^\infty\!\!\!\int_0^1\!\!\!\int_{\FF {E_\sigma}\cap\FF {F_\tau}}\nu_{f}\cdot\nu_{g}\,\dd\HH^h\,\dd\tau\,\dd\sigma\\
	&=\frac{1}{\sqrt{\pi}}\frac{\omega_{n-1}}{\omega_n}\int_0^\infty\!\!\!\int_0^1\!\!\!\int\nchi_{\partial^*F_\tau}(\nu_{f}\cdot\nu_{g})
	\,\dd{(\HH^h\mres(\FF E_\sigma))}\,\dd\tau\,\dd\sigma,
\end{split}\end{equation}
where the second equality is due to \cite[Lemma 3.27]{BG22}. 

Now recall that the map $(\tau,x)\mapsto \nchi_{\partial^* F_\tau}(x)$ is measurable with respect to $\LL^1\otimes (\HH^h \mres (\FF E_\sigma))$
(cf.\ with the proof of \cite[Proposition 3.33]{BG22}), so that by Fubini's theorem we can write the last term appearing in \eqref{eq:proof_main3_1} as 
\begin{equation}\label{eq:proof_main3_2}\begin{split}
&\frac{1}{\sqrt{\pi}}\frac{\omega_{n-1}}{\omega_n}\int_0^\infty\!\!\!\int (g^\vee-g^\wedge)(\nu_f\cdot\nu_{g})\,\dd{(\HH^h\mres(\FF E_\sigma))}\,\dd\sigma\\
=\,&\frac{1}{\sqrt{\pi}}\int_0^\infty\!\!\!\int (g^\vee-g^\wedge)(\nu_f\cdot\nu_{g})\,\dd|\DIFF\nchi_{E_\sigma}|\,\dd\sigma,
\end{split}\end{equation}
where we took into account \cite[Lemma 2.11]{ABP22} and finally \eqref{reprformula} for the last equality.
Notice that the integration over $\XX$ is only on $S_g$, which is a $\sigma$-finite set with respect to $\HH^h$,
so that by \cite[Theorem 5.3]{ambmirpal04} we know that $|\DIFF f|\mres S_g=|\DIFF f|\mres (S_f\cap S_g)$.
Now, by coarea we can continue the computation as 
\begin{equation}\label{eq:proof_main3_3}\begin{split}
&\frac{1}{\sqrt{\pi}} \int_{S_g} (g^\vee-g^\wedge)(\nu_f\cdot\nu_{g})\,\dd{(|\DIFF f|\mres S_f)}\\
=\,&\frac{1}{\sqrt{\pi}}\frac{\omega_{n-1}}{\omega_n} \int_{S_g} (g^\vee-g^\wedge)(f^\vee-f^\wedge)(\nu_f\cdot\nu_{g})\,\dd{(\HH^h\mres S_f)},
\end{split}\end{equation}
where we used \cite[Proposition 3.34]{BG22}. Combining \eqref{eq:proof_main3_1}, \eqref{eq:proof_main3_2}, and
\eqref{eq:proof_main3_3}, we finally obtain \eqref{eq:claim_main3}. Consequently, the proof of Theorem \ref{thm:main3} is complete.
\end{proof}
\end{document}